\documentclass[12pt,psamsfonts]{amsart}
\usepackage{amsmath,amsthm,amssymb,amscd,amsfonts,amsbsy}
\usepackage{subfigure,graphicx}
\usepackage{afterpage,rotating,pdflscape}
\usepackage{color}
\usepackage{cite} 
\usepackage{overpic}
\usepackage{epstopdf}
\usepackage{hyperref,url}
\usepackage{array,multirow,multicol}
\usepackage{adjustbox,blindtext}
\usepackage[dvipsnames]{xcolor}

\newcolumntype{R}{>{\displaystyle}r}
\newcommand{\R}{\ensuremath{\mathbb{R}}}
\newcommand{\CC}{\mathcal{C}}
\newcommand{\C}{\ensuremath{\mathcal{C}}}
\newcommand{\CO}{\ensuremath{\mathcal{O}}}
\newcommand{\V}{\ensuremath{\mathcal{V}}}

\newcommand{\CZ}{\ensuremath{\mathcal{Z}}}
\newcommand{\ov}{\overline}

\newcommand{\dis}{\displaystyle}
\newcommand{\wtilde}{\widetilde}

\newcommand{\f}{\varphi}
\newcommand{\al}{\alpha}
\newcommand{\la}{\lambda}
\newcommand{\be}{\beta}
\newcommand{\x}{\mathbf{x}}

\newcommand{\sgn}{\mathrm{sign}}

\newcommand{\inte}{\mathrm{int\,}}

\makeatletter
\newcommand{\tpitchfork}{%
	\raise-0.3ex\vbox{
		\baselineskip\z@skip
		\lineskip-.52ex
		\lineskiplimit\maxdimen
		\m@th
		\ialign{##\crcr\hidewidth\smash{$-$}\hidewidth\crcr$\pitchfork$\crcr}
	}%
}

\def\dis{\displaystyle}
\def\p{\partial}
\def\e{\varepsilon}

\newtheorem {theorem} {Theorem} 

\newtheorem {proposition} [theorem] {Proposition}

\newtheorem {lemma} [theorem] {Lemma}
\newtheorem {remark}[theorem]{Remark}

\newtheorem {mtheorem} {Theorem}

\begin{document}

\title[Codimension--two connection to a two-fold singularity]
{The generic unfolding of a codimension--two connection to a two-fold singularity of\\ planar Filippov systems}
\author[D.D.Novaes, M.A. Teixeira and I.O. Zeli]
{Douglas D. Novaes$^{1}$, Marco A.
Teixeira$^1$, and Iris O. Zeli$^{1}$}

\address{$^1$ Departamento de Matem\'{a}tica, Universidade
Estadual de Campinas, Rua S\'{e}rgio Buarque de Holanda, 651, Cidade
Universit\'{a}ria Zeferino Vaz, 13083--859, Campinas, SP, Brazil}
\email{ddnovaes@ime.unicamp.br}
\email{teixeira@ime.unicamp.br}
\email{irisfalkoliv@ime.unicamp.br} 



\maketitle

\begin{abstract}
	
 Generic bifurcation theory was classically well developed for smooth differential systems, establishing
results for $k$-parameter families of planar vector fields. In the present study we focus on a qualitative analysis of $2$-parameter families, $Z_{\al,\beta}$, of planar Filippov systems assuming that $Z_{0,0}$ presents a co\-di\-mension-two minimal set. 
Such object, named elementary simple two-fold cycle, is characterized by a regular trajectory connecting a visible two-fold singularity to itself, for which the second derivative of the first return map is nonvanishing. We analyzed the codimension-two scenario through the exhibition of its bifurcation diagram.
\end{abstract}

\section{Introduction}
 
Ongoing research in dynamical systems includes naturally nonsmooth systems, which commonly appear in realistic nonlinear engineering and control  models. As far as we know, the pioneering studies  of piecewise smooth systems in a rigorous way is due to Andronov and coworkers \cite{AVK}.  

In the 1970's, Qualitative and Geometric  theoretical analyses of two-fold
singularities of planar piecewise smooth systems have been taken into account in many studies, see for instance \cite{E,T}.  In addition,   in \cite{F} Filippov provided a mathematical formalization of the theory of nonsmooth vector fields. Since then, from various sides, attention has been paid to the generic classification of such singularities in the two dimensional case, see for instance \cite{KRG, GST, K}. More recently, other aspects of the two-fold singularity have been considered. For instance, in \cite{CLV} the problem of birthing of limit cycles from two-fold singularities (pseudo-Hopf bifurcation) was revisited, and in \cite{S} a probabilistic notion for the forward evolution from a two-fold singularity has been given  under small perturbations.

The main goal of this paper is to describe the bifurcation diagram of a closed trajectory connecting a two-fold singularity to itself.  We call this trajectory by {\it simple two-fold cycle} (see Figure \ref{fig0}). We emphasize that this cycle has a certain resemblance with the classical saddle homoclinic connection of smooth planar vector fields.
\begin{figure}[h] 
	\begin{center}
		\begin{overpic}[width=6cm]{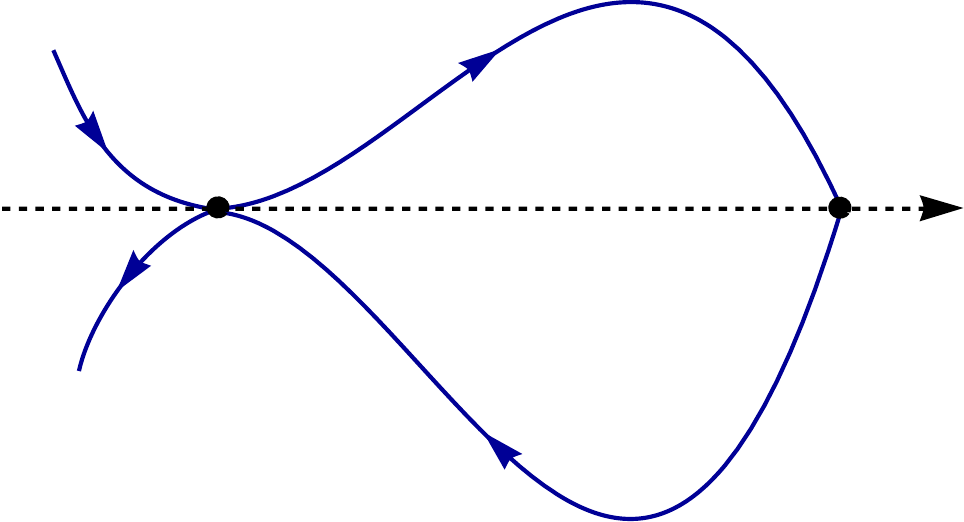}
		\end{overpic}
	\end{center}
	\caption{Simple two-fold cycle $\Gamma$.}
	\label{fig0}
\end{figure}

The study of bifurcations as well as the dynamics around invariant sets for smooth systems has classically been well
developed and extensively discussed over the years, mainly establishing conditions for the existence and persistence of
minimal sets (see \cite{SSTC,SSTC2,Rou}).  On the other hand, the theory of piecewise smooth differential systems has attracted considerable interest over the last decade (see, for instance,  \cite{BBCK,Var,CRM}, and references therein).  
In this direction many efforts are actually dedicated to understand  the dynamic behavior around some minimal sets when one finds no counterpart in the smooth world (see, for instance, \cite{AJMT,LH,NPV,NT}).

Regarding a simple two-fold cycle, there are several theoretical mathematical aspects related to it which are worthy of discussion. For instance, we mention branching of homoclinic cycles, periodic orbits, and heteroclinic trajectories, their stability properties, and the exhibition of its bifurcation diagram. Furthermore, the needed tools to analyze the simple two-fold cycle go beyond the use of the {\it Poincar\'{e} Map}.
We emphasize that this present study has been mainly motivated by these theoretical aspects.  For more on bifurcations in piecewise smooth systems we may refer to the book \cite{simpson}.

As far as we are concerned, there are no works on nonsmooth physical phenomena for which the corresponding nonsmooth mathematical models exhibit simple two-fold cycles. However, we were able to find an example of a piecewise mechanical system having this kind of cycle (see Section \ref{ppm}). Hence, we hope that the results we have obtained in the present paper may be useful in the future to better understand some real phenomena.

This paper is organized as follows. Section \ref{prel} contains some basic concepts on nonsmooth theory as well as the formal definition of a simple two-fold  cycle (see Figure \ref{fig0}). Our main results are stated in Section \ref{dnc}: Theorem \ref{R2} provides a non-degeneracy condition under which a simple two-fold cycle has codimension-two; and Theorem \ref{T:BD} describes the bifurcation diagram of a simple two-fold cycle provided the previous non-degeneracy condition. Section \ref{proofs} contains some preliminary results needed to prove our main theorems. Sections \ref{proofA} and \ref{proofB} are devoted to prove Theorems \ref{R2} and \ref{T:BD}, respectively. In Section \ref{ppm} we study a 2-parameter family of piecewise Hamiltonian differential systems realizing the bifurcation diagram given by Theorem \ref{T:BD}. Finally, in Section \ref{conclusion} some closing remarks and further directions are provided.

\section{Basic notions on Filippov systems} \label{prel}
In this section we briefly introduce the basic notions on piecewise smooth planar differential systems. For more details see \cite{F,GST}.

Let $U \subset \R^2$ be an open bounded set containing the origin $(0,0)$. Given $r\geq 1$, let $\chi^r$ be the set of all $\CC^r$ vector fields $X: U  \to \R^2$ endowed with the $\C^r-$topology. Consider a  $\CC^1$ function $h:U\rightarrow\R,$ for which $0\in\R$ is a regular value, and let $\Omega^r$ be the space of the following piecewise vector fields\begin{equation}\label{s11}
Z(x,y)=\left\{\begin{array}{l} X(x,y) ~ \textrm{if} ~ h(x,y)>0,\vspace{0.2cm}\\
Y(x,y) ~ \textrm{if}~ h(x,y)<0,
\end{array}\right.
\end{equation} 
where  $X,Y\in \chi^r$,  $(x,y)\in U$  and $\Sigma=h^{-1}(0)$ is the switching manifold.  So $\Omega^r = \chi^r \times \chi^r$ can be endowed with the product topology. Accordingly, we denote $Z=(X,Y)$. When the context is clear we write $\chi^r=\chi$ and $\Omega=\Omega^r$. 

For each $X \in \chi$ we define the smooth function
$Xh: U \to \R$ given by $Xh=X\cdot \nabla h$, where $\cdot$ is the  canonical scalar product in $\R^2$. As usual, for $Z=(X,Y) \in \Omega$, we distinguish three different open regions in $\Sigma$: the {\it sliding region} $\Sigma^s$ (resp. {\it escaping region} $\Sigma^e$) satisfying $Xh(x,y) < 0$ and $Yh(x,y) > 0$ (resp. $Xh(x,y) > 0$ and $Yh(x,y) < 0$ ), and the {\it crossing region} $\Sigma^c$ satisfying $Xh(x,y) Yh(x,y) > 0$. The boundaries of the regions $\Sigma^c$, $\Sigma^s$, and $\Sigma^e$  (i.e. $Xh(x,y)Yh(x,y)  =0$) are constituted by tangency points of $X$ or $Y$ with the switching manifold $\Sigma$.

Here we assume that the solutions of the piecewise smooth differential system $ (x',y') = Z(x,y)$,  for $Z\in \Omega$, obeys Filippov's convention (see \cite{F}). In this case, the piecewise vector field \eqref{s11} is called {\it Filippov vector field}. We recall that when  $p \in \Sigma^s \cup \Sigma^e$ the local trajectory of $Z\in\Omega$ through $p$ follows the trajectory of  the so-called {\it sliding vector field}
\begin{equation}
\label{svfz}
Z^s(p)=\dfrac{Y h(p) X(p)-X h(p) Y(p)}{Y h(p)- Xh(p)}.
\end{equation}
Notice that, for $p\in\Sigma$, $Z^s(p)$ is tangent to $\Sigma$ at $p$. We say that a point $p \in \Sigma^s \cup \Sigma^u$ is a {\it pseudo equilibrium} of $Z$ if it is an equilibrium of $Z^s,$ that is, $Z^s(p)=0$. When $Z^s$ is defined in an open region $V \subset \Sigma^s \cup \Sigma^e$ with boundary $\p V$, it can be $\CC^r$-extended to a full neighborhood of $p$ for all $ p \in \p V$ in $\Sigma$.  

A point $p\in U$ is said to be a singularity of the Filippov vector field \eqref{s11} if either $(i)$ $p$ is a singularity of $X$ or $Y$ (i.e. $X(p)=0$ or $Y(p)=0$), or $(ii)$ $p\in \Sigma^s\cup\Sigma^e$ is a pseudo equilibrium (i.e. $Z^s(p)=0$), or $(iii)$ $p$ is a tangency point (i.e. $p\in \p \Sigma^c\cup\p \Sigma^s\cup \p\Sigma^e$). Otherwise $p$ is called a {\it regular point}. Particularly, in the case $(iii)$, we say that a point $p\in \Sigma$ is a {\it visible fold } for  $X$ (resp. $Y$), if  $Xh(p)=0$ and $X (Xh) (p)> 0$ (resp. $Yh(p)=0$ and $Y (Yh) (p) < 0$).  Analogously, reversing the inequalities, we define an {\it invisible fold}. A point $p\in \Sigma$ is called {\it visible two-fold singularity} (resp. {\it invisible two-fold singularity}) of $Z=(X,Y)$ if it is a visible (resp. invisible) fold for $X$ and $Y,$ simultaneously. In addition, a point $p\in \Sigma$ is called {\it regular-fold singularity} of $Z$  if  it is a fold point for $X$  and a regular point for $Y$, or vice versa.

\subsection{Simple Two-Fold Cycle} 
Let $Z=(X,Y)\in\Omega$ be a piecewise smooth vector field  as defined in \eqref{s11}.  A {\it simple two-fold cycle} is characterized by a trajectory $\gamma(t)$ of $Z$ for which there exists $T>0$ such that $\gamma(0)=\gamma(T)=p$ is a visible two-fold singularity of $Z$, and, for each $0<t<T$, $\gamma(t)$ is a regular point of $Z$ (see Figure \ref{ciclos}).

\begin{figure}[h]
	\subfigure[]
	{\begin{overpic}[width=3.5cm]{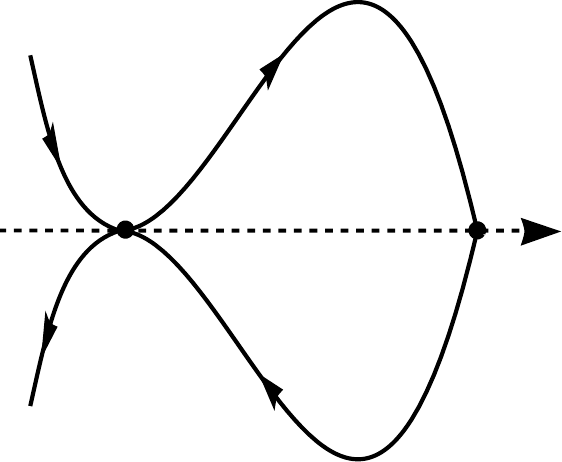}
	\end{overpic}}
	\hspace*{1cm}
	\subfigure[]
	{\begin{overpic}[width=3.4cm]{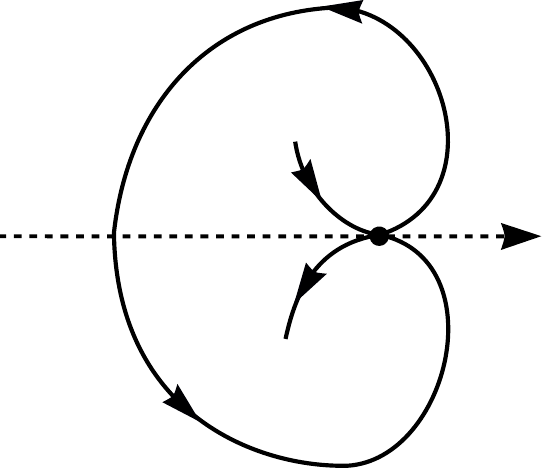}
	\end{overpic}}
	\hspace*{1cm}
	\subfigure[]
	{\begin{overpic}[width=3.5cm]{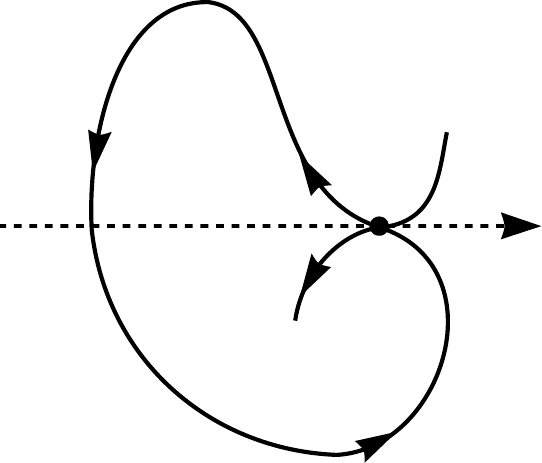}
	\end{overpic}}
	\caption{Examples of simple two-fold cycles.}
	\label{ciclos}	
\end{figure}
 Notice that for the examples illustrated by Figures \ref{ciclos}(a) and \ref{ciclos}(b) a first return map may be defined in a half-closed interval $[p,p+a)$ contained in $\Sigma\cap\textrm{int}(\gamma)$ and $\Sigma\cap\textrm{ext}(\gamma)$, respectively. Nevertheless, the example illustrated by Figure \ref{ciclos}(c) does not admit a non degenerate first return map. Indeed, in this last case, if a first return map is well defined, then it is constant. Accordingly, a simple two-fold cycle, for which a first return map is well defined, is called {\it elementary} whenever the second derivative of its first return map is nonvanishing. This property implies that the orbits lying in a small annulus of $\Gamma$ are not closed.

 In this paper our attention will be focused on the cycle illustrated by Figure \ref{ciclos}(a). The analysis of any other elementary two-fold cycle can be performed in an analogous way. Throughout this paper, without loss of generality, we shall take $h(x,y)=y$. Denote by $X^1$, $X^2$ and $Y^1$, $Y^2$ the coordinates of $X$ and $Y$, respectively.  In order to characterize this simple two-fold cycle we need to assume the following condition:

\begin{itemize}
	\item[$(C)$] The vector field $Z =(X , Y )$ has a visible two-fold point at $(0,0)\in\Sigma$ such that $X^1(0,0)>0$ and $Y^1(0,0)<0$. The trajectory of  $X $ (resp. $Y$) passing through $(0,0)$ meets $\Sigma$ transversally at $(q_{X },0)$ (resp. $(q_{Y },0)$) forward in time (resp. backward in time), where $q_{X }=q_{Y }>0$. 
\end{itemize}	

Note that the orientation of the trajectories of $X $ and $Y $ are fixed by $(C)$. More precisely, the trajectory of $X $ 
 goes to the right and the  trajectory of $Y $ goes to
the left, so that system \eqref{s11} admits a cycle $\Gamma$, which is characterized by the union $\Gamma=\Gamma_X\cup \Gamma_Y\cup\{(0,0)\}$, where 
\begin{equation*}
\label{Ga}
\Gamma_{X} =\{\varphi_{X }(t,0,0), ~0 < t \leq T^+ \} ~ \text{and} ~ 
\Gamma_{Y} =\{\varphi_{Y }(t,0,0), ~T^-\leq t < 0 \}.
\end{equation*}
Here $\varphi_{X }(t,x,y)$ and $\varphi_{Y }(t,x,y)$  are the trajectories of $X $ and $Y $, respectively, satisfying $\varphi_{X }(T^+,0,0)=(q_{X },0)$ and $\varphi_{Y }(T^-,0,0)=(q_{Y },0)$ (see Figure \ref{fig0}). 

For $\delta>0$ sufficiently small and $\sigma=\{(x,0):0<x<\delta\}\subset\Sigma$,  it is well defined a {\it displacement function} $f_{Z}:\sigma\rightarrow\Sigma$ associated with $Z$, defined by $f_{Z}(x)=\varphi_{X}(T_1(x),x,0) -\varphi_{Y}(T_2(x),x,0)$ where $T_1(x)>0$ is the smallest positive time such that $\varphi_{X}(T_1,x,0)\in\Sigma$, and $T_2(x)<0$ is the biggest negative time such that  $\varphi_{Y}(T_2(x),x,0) \in \Sigma $ (see Figure \ref{fig1}).  In Section \ref{subdisp} (see Proposition \ref{itemi}), we shall prove that 
\begin{equation}\label{fZ}
f_{Z}(x) = M x^2 + \CO(x^3), ~ M \in \R.
\end{equation}
We recall that the notation $u(x)=\CO(v(x))$ means that there exist constants $d>0$ and $K>0$ such that $|u(x)|<K |v(x)|$ whenever $|x|<d$. Notice that the cycle $\Gamma$ is {\it elementary} provided that $M\neq 0$. Moreover, if $M>0$ the cycle $\Gamma$ is stable, and if $M<0$ the cycle $\Gamma$ is unstable. 

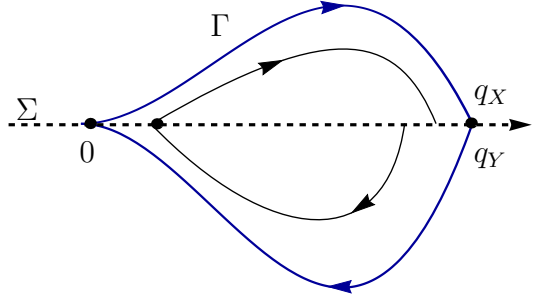
\begin{figure}[h]
	\begin{center}
		\begin{overpic}[width=7cm]{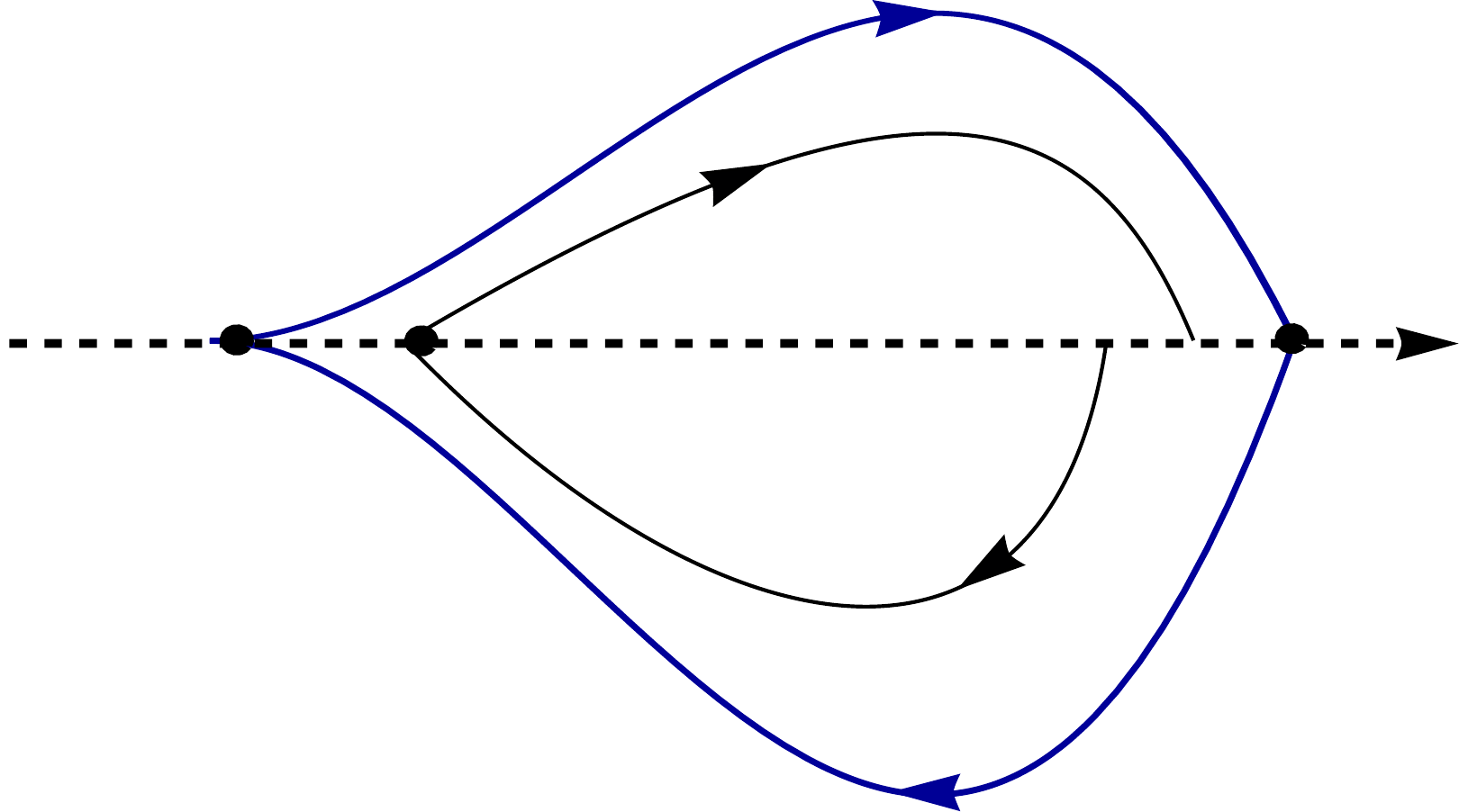}
			\put(2,33){$\Sigma$} 
			\put(14.1,25){$0$}
			\put(39,49){$\Gamma$}
			\put(89,37){$q_{X}$}
			\put(89,25){$q_{Y}$}
		\end{overpic}
	\end{center}
	\caption{Elementary simple two-fold cycle $\Gamma$  connecting the visible two-fold $(0,0)$ to itself and reaching $\Sigma$  transversally at $(q_{X},0)=(q_{Y},0)$. The elementary property implies that the orbits lying in a small annulus of $\Gamma$ are not closed.}
	\label{fig1}
\end{figure}

\section{Main results}\label{dnc}

Our main goal in this paper is to understand what typically happens when an elementary simple two-fold cycle is perturbed on $\Omega$. More precisely, let $\Gamma_0$ be a simple two-fold cycle of a Filippov vector field $Z_0=(X_0,Y_0)\in\Omega$ characterized by condition $(C),$ and let $\mathcal{A}_0\subset\R^2$ be a sufficiently small annulus around $\Gamma_0.$ Considering Filippov systems lying in a small neighborhood $\V_0\subset\Omega$ of $Z_0$ we may ask how their phase spaces look like on $\mathcal{A}_0$ (see Figure \ref{fig2}). A complete characterization of  these systems will be given by Theorems \ref{R2} and \ref{T:BD}, assuming that $\Gamma_0$ is elementary.

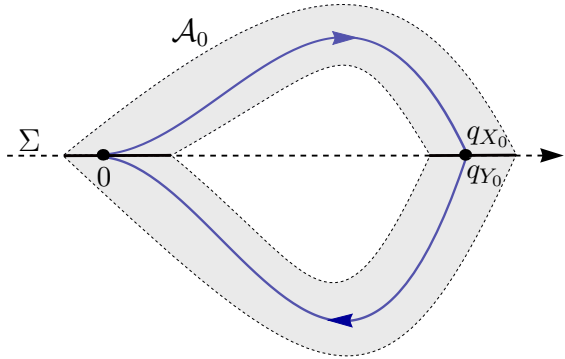
\begin{figure}[h]
	\begin{center}
		\begin{overpic}[width=7.4cm]{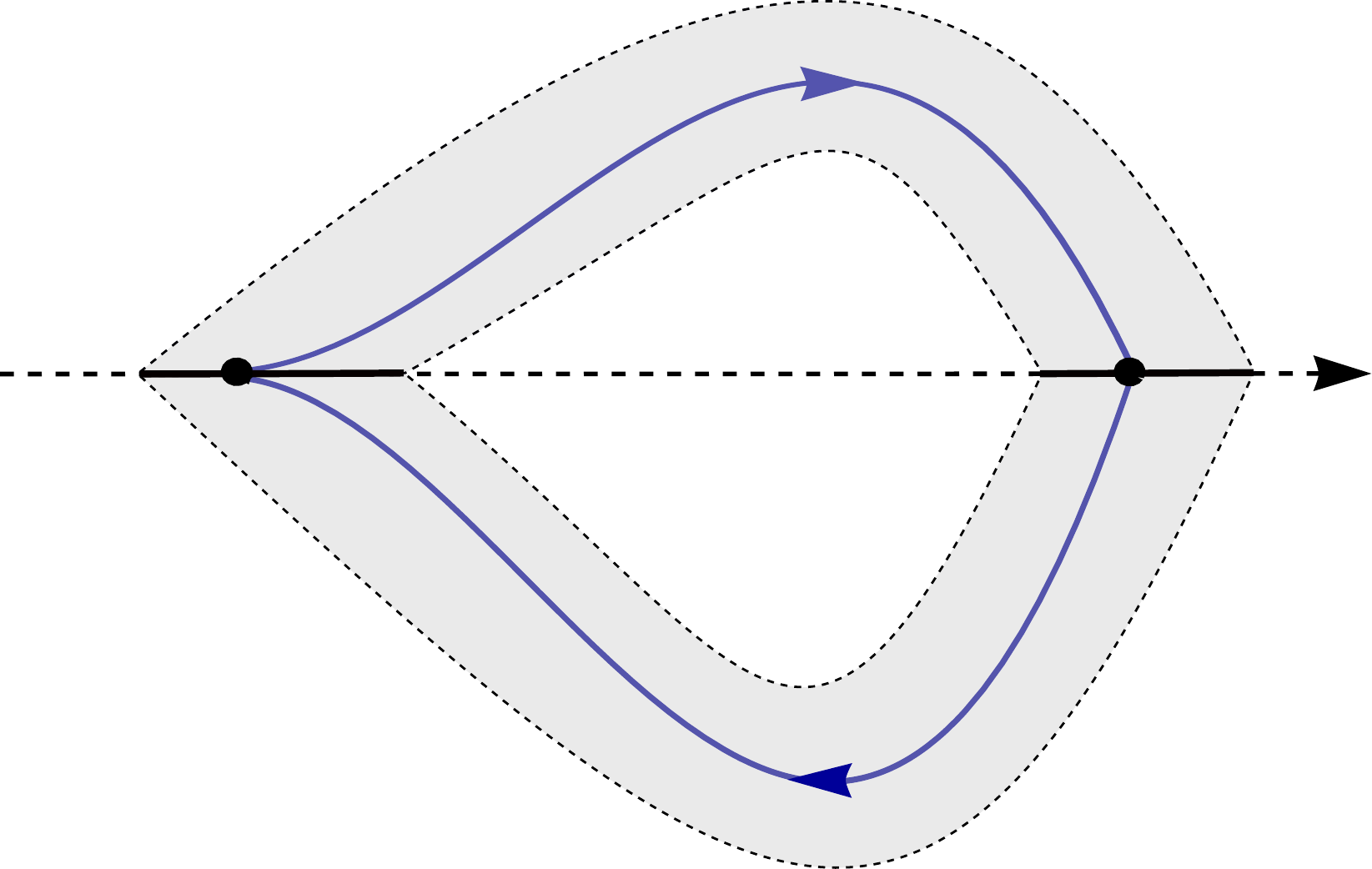}
			\put(2,37){$\Sigma$} 
			\put(29.5,56.7){$\mathcal{A}_0$}
			\begin{small}
				\put(16.1,30.7){$0$} 
				\put(82.5,39){$q_{X_0}$}
				\put(82.5,32){$q_{Y_0}$}
			\end{small}
		\end{overpic}
	\end{center}
	\caption{Small annulus $\mathcal{A}_0 \subset \R^2$ around the simple two-fold cycle $\Gamma_0$.}
	\label{fig2}
\end{figure}

\subsection{Perturbation of a simple two-fold cycle}
Assume that $\Gamma_0$ is a simple two-fold cycle of $Z_0\in\Omega$ characterized by condition $(C)$. Given an annulus $\mathcal{A}_0\subset\R^2$ around $\Gamma_0$, there exists a neighborhood $\V_0\subset\Omega$ of $Z_0$ such that each $Z=(X,Y)\in\V_0$ admits a fold point $(p_{X},0)\in \mathcal{A}_0$ of $X$, and a fold point $(p_{Y}, 0)\in \mathcal{A}_0$ of $Y$. Indeed, $X_0$ and $Y_0$ are $C^r$ vector fields, with $r\geq 1$, and the neighborhood $\V_0$ can be taken in $\chi^r \times \chi^r$ as a Cartesian product of neighborhoods in $\chi^r$ of $X_0$ and $Y_0$, respectively. So, that the above statement follows  from the continuous dependence of solutions on initial conditions and parameters.

In this case, denoting by $X^1,X^2$ and $Y^1,Y^2$ the coordinates of $X$ and  $Y$, respectively, we have
\begin{equation}
\label{p1}
\begin{array}{lll}
X^2(p_X,0)=0, & \dfrac{\p X^2}{\p x }(p_X,0)> 0, &  X^1(p_X,0)>0,\vspace{0.2cm}\\
Y^2(p_{Y},0)=0, & \dfrac{\p Y^2}{\p x }(p_{Y},0)> 0, & Y^1(p_{Y},0) < 0.
\end{array}
\end{equation}
Furthermore,  condition $(C)$ implies that the neighborhood $\V_0$ can be chosen such that: 
\begin{itemize}
\item[$(i)$] the trajectory of  $X$, forward in time, starting at $(p_{X},0)$ meets $\Sigma$ transversally at $(q_{X},0),$ that is,  $X^2(q_{X},0) \neq 0$; 
\item[$(ii)$] the trajectory of  $Y$, backward in time, starting at $(p_Y,0)$ meets $\Sigma$ transversally at $(q_{Y},0),$ that is, $Y^2(q_{Y},0) \neq 0$;
\item[$(iii)$] the trajectory of $X$ connecting $(p_X,0)$ to $(q_X,0)$ and the trajectory of $Y$ connecting $(p_Y,0)$ to $(q_Y,0)$ are both contained in $\mathcal{A}_0$.
\end{itemize}

Without loss of generality we assume that $p_{X}=0$ for all $Z \in \V_0$ $($see Figure \ref{psystem}$)$.  
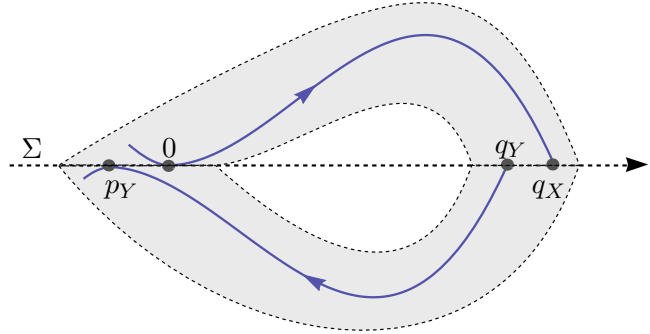
\begin{figure}[h] 
	\begin{center}
\begin{overpic}[width=8.5cm]{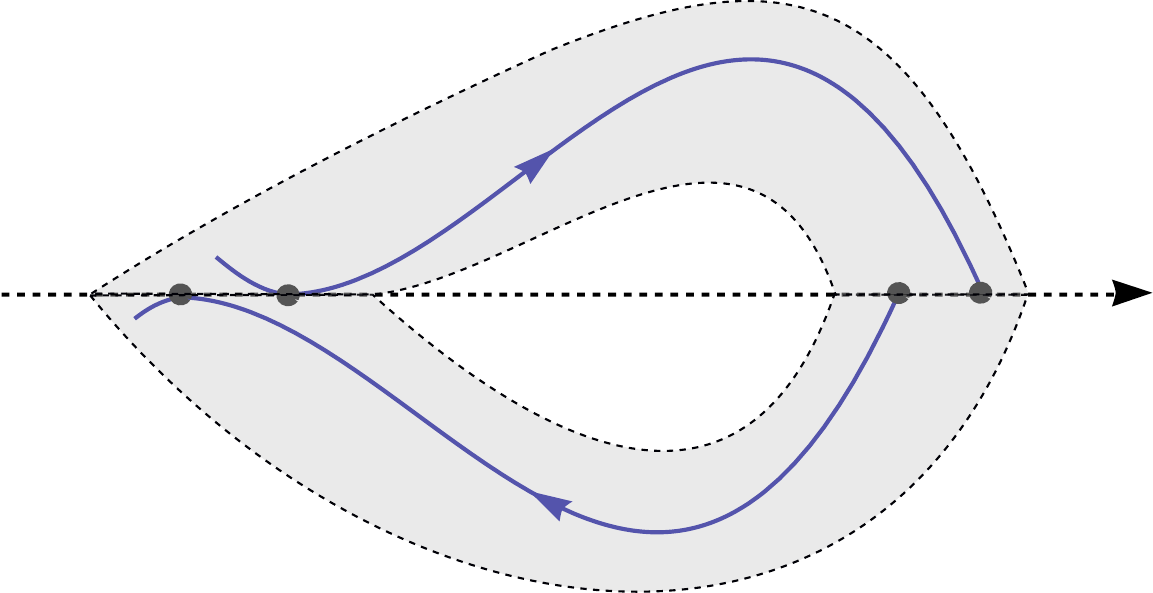}
				\begin{small}
			\put(2,27){$\Sigma$}
			\put(15,21.5){$p_{Y}$}
			\put(24,27.2){$0$}
			\put(76,27.9){$q_{Y}$}
			\put(81.8,21.5){$q_{X}$}
			\end{small}
		\end{overpic}
	\end{center}
	\caption{Trajectories of $X$ and $Y$ through $(0,0)$ and $(p_{Y},0)$, respectively.}
	\label{psystem}
\end{figure}

The above comments allow us to define a $\C^r$ function $\eta: \V_0 \to \R^2$ by 
\begin{equation}
\label{bifun}
\eta(Z)=(p_{Y},q_{X} - q_{Y}). 
\end{equation} 
Note that $\eta(Z_0)=(0,0)$.  Our first main result states that the function $\eta$ is a submersion, that is, the derivative $d\eta(Z):\Omega^r\rightarrow \R^2$ is a surjective linear map for every $Z\in\V_0$. It implies that $\eta^{-1}(0,0)$ is a codimension-two submanifold of $\V_0$ (see \cite{L}).

\begin{mtheorem}
\label{R2} 
Let $\Gamma_0$ be an elementary simple two-fold cycle of a vector field $Z_0 \in \Omega$ characterized by condition $(C)$. Then there exist an annulus $\mathcal{A}_0\subset\R^2$ around $\Gamma_0$ and a neighborhood $\V_0\subset\Omega$ of $Z_0$  for which the following statements hold: $(i)$ for $Z\in \V_0$, $\eta(Z)=(0,0)$ if and only if $Z$ has an elementary simple two-fold cycle in $\mathcal{A}_0$, and $(ii)$ $d\eta(Z)$ is surjective for each $Z\in \V_0$.  
\end{mtheorem} 
Theorem \ref{R2} is proved in Section \ref{proofA}. 

\subsection{Bifurcation Diagram}

Theorem \ref{R2} implies that an elementary simple two-fold cycle $\Gamma_0$ persists in a codimension-two submanifold of $\V_0$. Furthermore, since $\eta$ is a submersion, the cycle $\Gamma_0$ can be unfolded using the function $\eta$ so that all the bifurcations occurring in a small annulus of $\Gamma_0$ can be detected.

Accordingly, let $U=\eta(\V_0)\subset\R^2$ and  $(\al,\beta)=\eta(Z)$ for $Z\in \V_0$. As the parameter $\alpha$ varies both folds move apart (see Figure \ref{psystem}) creating a sliding region between them, which contains a pseudo equilibrium. Meanwhile, as the parameter $\beta$ varies several bifurcations may occur depending on the sign of $\al$.  For the sake of simplicity, let us assume that $\Gamma_0$ is stable. 

For $\al>0$, in addition to the curve $\be=0$, we may find other curves of codimension-one bifurcations:
\begin{equation}\label{curvecon1}
\be_1(\al)>\be_2(\al)>\be_4(\al)>0\quad \text{with} \quad \be_4(0)=\be_2(0)=\be_1(0)=0.
\end{equation}  
We shall see that for $\be=0$ the vector field  $Z$ has a connection between visible regular-fold singularities. The curves $\beta=\beta_1(\alpha)$ and  $\beta=\beta_2(\alpha)$  represent a saddle-node bifurcation curve and a unstable critical crossing bifurcation curve (see \cite{FPT}), respectively. Finally, for $\beta=\beta_4(\alpha)$ the vector field  $Z$ has a connection between a stable pseudo equilibrium and a regular-fold singularity.
	
On the other hand, for $\alpha<0$, in addition to the curve $\be=0$, we may also find other curves of codimension-one bifurcations:
\begin{equation}\label{curvecon2}
0<\beta_3(\alpha)<\beta_5(\al) \quad \text{with} \quad \beta_3(0)=\beta_5(0)=0.
\end{equation}   
We shall see that for $\be=0$ the vector field  $Z$ still has a connection between visible regular-fold singularities. The curve $\beta=\beta_3(\alpha)$ is a stable critical crossing bifurcation curve, and for $\beta=\beta_5(\alpha)$  the vector field  $Z$ has a connection between a unstable pseudo equilibrium and a regular-fold singularity.

From Theorem $A$, the set $U$ is an open neighborhood of $(0,0)$. In what follows we define some regions (see Figure \ref{spacepar}):
\begin{equation}
\begin{array}{l}
\label{regions}
R_1=\{(\alpha,\beta)\in U: \alpha<0~\text{and}~ \beta < \beta_3(\al) \};\\
R_2=\{(\alpha,\beta)\in U: \alpha>0~\text{and}~\beta < 0 \};\\
R_3=\{(\alpha,\beta)\in U: \alpha>0~\text{and}~0<\beta< \beta_4(\al) \};\\
R_4=\{(\alpha,\beta)\in U: \alpha>0~\text{and}~\beta_4(\al)<\beta< \beta_2(\al)\};\\
R_5=\{(\alpha,\beta)\in U: \alpha>0~\text{and}~\beta_2(\al)<\beta< \beta_1(\al)\};\\
R_6=\{(\alpha,\beta)\in U: \alpha>0~\text{and}~ \beta > \beta_1(\al) \};\\
R_7=\{(\alpha,\beta)\in U: \alpha<0~\text{and}~\beta>0\};\\
R_8=\{(\alpha,\beta)\in U: \alpha<0~\text{and}~\beta_5(\al)<\beta< 0 \};\\
R_9=\{(\alpha,\beta)\in U: \alpha<0~\text{and}~\beta_3(\al)<\beta< \beta_5(\al)\}.
\end{array}
\end{equation}
For the sake of convenience we denote by $B_{ij}=(\p R_i\cap\p R_j)\setminus\{(0,0)\}$ the boundary between $R_i$ and $R_j$, for $i,j\in\{1,2,\ldots,9\}$. In this case, the curves $\beta=\beta_i(\alpha)$, for $i \in \{ 1, 2, \ldots, 5\}$,  correspond to $B_{56}$, $B_{45}$, $B_{91}$, $B_{34}$, and $B_{89}$, respectively (see Figure \ref{spacepar}).

\begin{figure}[h] 
	\begin{center}
	\begin{overpic}[width=6cm]{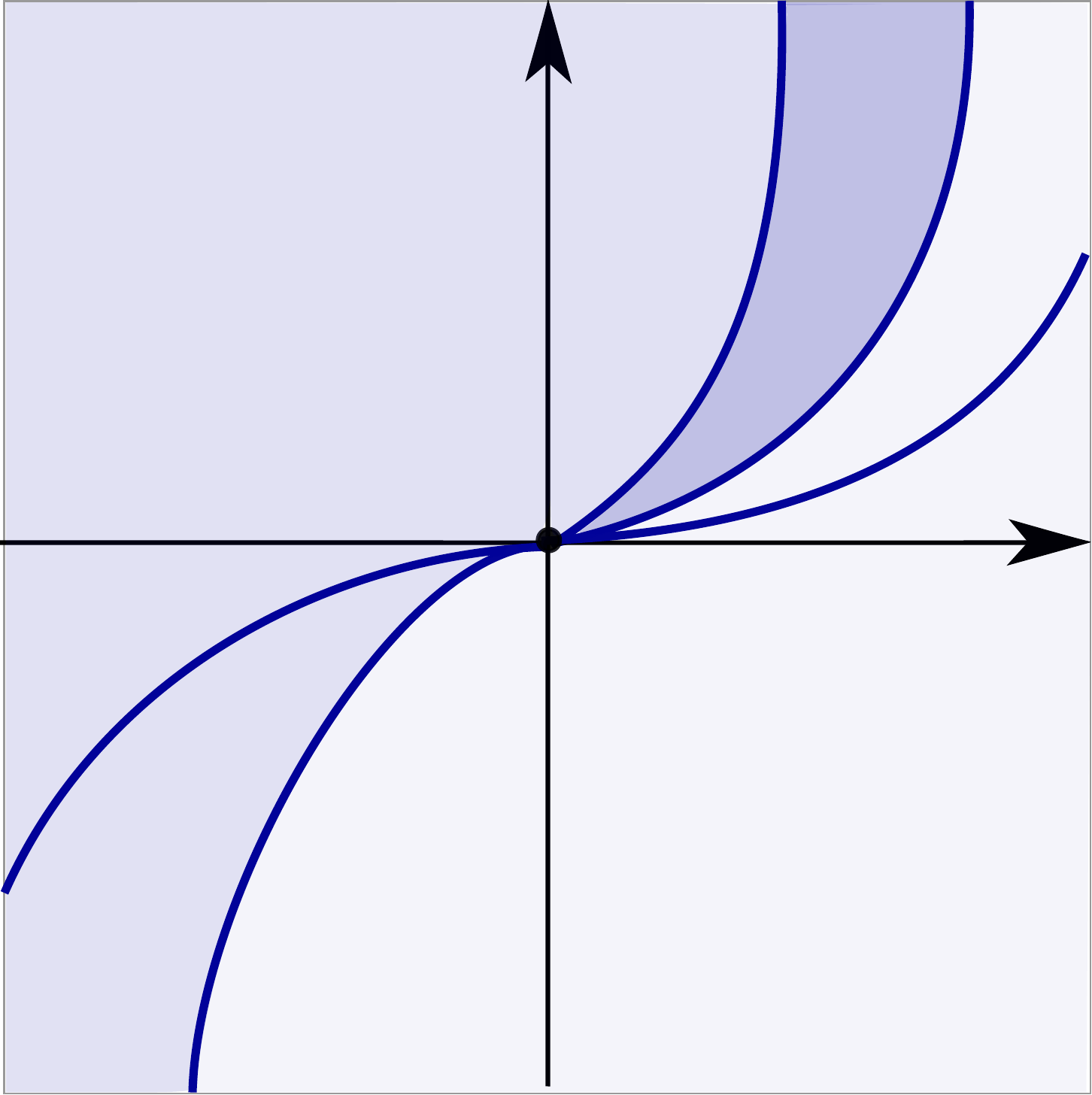}
			\begin{footnotesize}
				\put(95,44){$\alpha$}
				\put(53,94){$\beta$}
				\put(67,103){$ B_{56}$}
				\put(84,103){$ B_{45}$}
				\put(101,75){$ B_{34}$}
				\put(14,-6){$ B_{91}$}
				\put(45,103){$ B_{67}$}
				\put(101,48){$ B_{23}$}
				\put(-12,48){$ B_{78}$}
				\put(-12,17){$ B_{89}$}
				\put(46,-6){$ B_{12}$}
                \put(35,18){$ R_1$}
				\put(73,18){$ R_2$}
                \put(85,54){$ R_3$}
				\put(88,78){$ R_4$}
				\put(73,78){$ R_5$}
                \put(55,78){$ R_6$}
				\put(20,75){$ R_7$}
           		\put(6,42){$ R_8$}
				\put(13,27){$R_9$}
				\begin{footnotesize}
				\put(72.3,94){$  {\beta_1}$}
				\put(89.5,94){$ \beta_2$}
				\put(94,65){$ \beta_4$}
				\put(19,2){$ \beta_3$}
				\put(2,17){$ \beta_5$}
				\end{footnotesize}
			\end{footnotesize}
		\end{overpic}
	\end{center}
	\caption{Space of parameters $(\alpha,\beta)$ and codimension-one bifurcation curves. Here $\beta_i$ represents the curves $\beta=\beta_i(\alpha)$ for $i\in \{ 1,2 \ldots, 5\}$.} 
	\label{spacepar}
\end{figure}
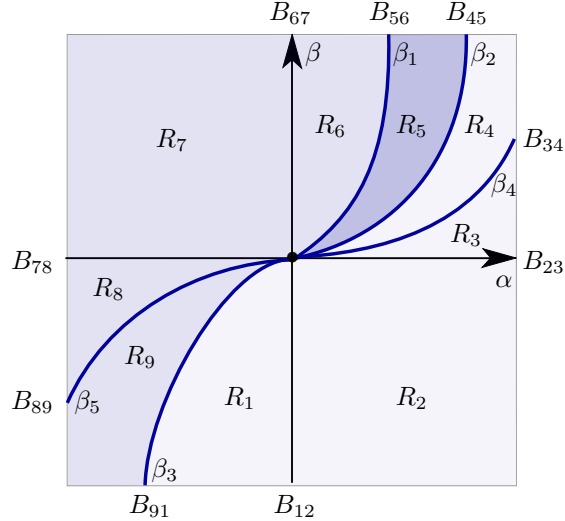

Our second main result describes completely the behavior of  vector fields $Z\in\Omega$ nearby $Z_0$, restricted to a small annulus of $\Gamma_0$. It is a descriptive version of the bifurcation digram illustrated in Figure \ref{stable}. 

\begin{mtheorem}\label{T:BD} Assume that $\Gamma_0$ is a stable elementary simple two-fold cycle of a vector field $Z_0 \in \Omega$. Let $\mathcal{A}_0\subset\R^2$  and $\V_0\subset\Omega$ be the neighborhoods of $\Gamma_0$ and $Z_0$, respectively, given by Theorem \ref{R2}. Then there exist an annulus $\mathcal{A}_1\subset\mathcal{A}_0$ of $\Gamma_0$, a neighborhood $\V_1\subset\V_0$ of $Z_0$, and curves $\beta_i(\al)$, $i\in\{1,2,\ldots,5\}$ satisfying \eqref{curvecon1} and \eqref{curvecon2} such that,  for $Z\in\V_1$ and $(\al,\beta)=\eta(Z)$, the following possibilities for $Z|_{\mathcal{A}_1}$ hold:
\begin{itemize}
	\item [$(a)$]  for $(\alpha,\beta) \in R_1$, there exist two regular-fold points, $(0,0)$ and $(\al,0)$, a unstable pseudo equilibrium $(p^u,0)$, and a stable crossing cycle $\gamma_1$;
	\item [$(b)$]  for $(\alpha,\beta) \in B_{12}$, there exist a two-fold point at $(0,0)$ and a stable crossing cycle $\gamma_1$, with  $(0,0)\in\textrm{ext}(\gamma_1)$;
	\item [$(c)$]  for $(\alpha,\beta) \in R_2$, there exist two regular-fold points, $(0,0)$ and $(\al,0)$, a stable pseudo equilibrium $(p^s,0)$, and a stable crossing cycle $\gamma_1$;
	\item [$(d)$]  for  $(\alpha,\beta) \in B_{23}$,  there exist two regular-fold points, $(0,0)$ and $(\al,0)$, a stable pseudo equilibrium $(p^s,0)$, a stable crossing cycle $\gamma_1$, and a connection between $(0,0)$ and $(\al,0)$;	
	\item [$(e)$]  for  $(\alpha,\beta) \in R_{3}$,  there exist two regular-fold points $(0,0)$ and $(\al,0)$, a stable pseudo equilibrium $(p^s,0)$, a stable crossing cycle $\gamma_1$, and a sliding connection between $(0,0)$ and $(\al,0)$ contained in $\textrm{ext}(\gamma_1)$;
	\item [$(f)$] for  $(\alpha,\beta) \in B_{34}$,  there exist two regular-fold points, $(0,0)$ and $(\al,0)$, a stable pseudo equilibrium $(p^s,0)$, a stable crossing cycle $\gamma_1$, and a connection between $(p^s,0)$ and $(\al,0)$ contained in $\textrm{ext}(\gamma_1)$;
	\item [$(g)$]  for  $(\alpha,\beta) \in R_4$,  there exist two regular-fold points, $(0,0)$ and $(\al,0)$, a stable pseudo equilibrium $(p^s,0)$, a stable crossing cycle $\gamma_1$, and a unstable sliding cycle contained in $\textrm{ext}(\gamma_1)$ and passing through $(\alpha,0)$.
	\item [$(h)$]  for $(\alpha,\beta) \in B_{45}$, there exist two regular-fold points, $(0,0)$ and $(\al,0)$, a stable pseudo equilibrium $(p^s,0)$, a stable crossing cycle $\gamma_1$, and a unstable critical crossing cycle contained in $\textrm{ext}(\gamma_1)$ and passing through $(\alpha,0)$;	
	\item [$(i)$]  for $(\alpha,\beta) \in  R_5$, there exist two regular-fold points, $(0,0)$ and $(\al,0)$, a stable pseudo equilibrium $(p^s,0)$,  a stable crossing cycle $\gamma_1$, and a unstable crossing cycle $\gamma_2$ contained in $\textrm{ext}(\gamma_1)$;	
	\item [$(j)$]  for $(\alpha,\beta) \in  B_{56}$, there exist two regular-fold points, $(0,0)$ and $(\al,0)$, a stable pseudo equilibrium $(p^s,0)$, and a semi-stable crossing cycle $\gamma_3$, which attracts the orbits contained in $\textrm{int}(\gamma_3)\cap\mathcal{A}_{1}$ and repels the orbits contained in $\textrm{ext}(\gamma_3)\cap\mathcal{A}_{1}$;	
	\item [$(k)$]  for $(\alpha,\beta) \in  R_6$, there exist two regular-fold points, $(0,0)$ and $(\al,0)$, a stable pseudo equilibrium $(p^s,0)$, and there is no cycle;
   	\item [$(l)$]  for $(\alpha,\beta) \in  B_{67}$, there exists a two-fold point at $(0,0)$.  
   	\item [$(m)$]  for $(\alpha,\beta) \in R_7$, there exist two regular-fold points, $(0,0)$ and $(\al,0)$, and a unstable pseudo equilibrium $(p^u,0)$.
	\item [$(n)$]  for $(\alpha,\beta) \in  B_{78}$, there exist two regular-fold points, $(0,0)$ and $(\al,0)$, a unstable pseudo equilibrium $(p^u,0)$, and a connection between $(0,0)$ and $(\al,0)$;	
	\item [$(o)$]  for $(\alpha,\beta) \in R_8$, there exist two regular-fold points $(0,0)$ and $(\al,0)$, a unstable pseudo equilibrium $(p^u,0)$, and a sliding connection between $(0,0)$ and $(\al,0)$;
	\item [$(p)$]  for $(\alpha,\beta) \in  B_{89}$,  there exist two regular-fold points, $(0,0)$ and $(\al,0)$, a unstable pseudo equilibrium $(p^u,0)$, and a connection between  $(0,0)$ and $(p^u,0)$;
	\item [$(q)$]  for  $(\alpha,\beta) \in R_9$,  there exist two regular-fold points, $(0,0)$ and $(\al,0)$, a unstable pseudo equilibrium $(p^u,0)$, and a stable sliding cycle passing through $(0,0)$.
	\item [$(r)$]  for $(\alpha,\beta) \in B_{91}$, there exist two regular-fold points, $(0,0)$ and $(\al,0)$, a unstable pseudo equilibrium $(p^u,0)$, and a stable critical crossing cycle passing through $(0,0)$;
\end{itemize}
Moreover,
\[
\beta_1 (\alpha)= \frac{k \ell}{L} \alpha^2 + \CO(\alpha^3),~ \beta_2(\alpha)=- \ell \alpha^2+\CO(\alpha^3),~\text{and} ~\beta_3(\alpha)=k \alpha^2+\CO(\alpha^3).
 \]
\end{mtheorem}

The diagram bifurcation for the case when $\Gamma_0$ is a unstable elementary simple two-fold cycle of a vector field $Z_0$  is obtained by rotating the previous one (see Figure \ref{stable}) by $\pi$ radius (see Figure \ref{repulsive}).

\begin{figure}
	\vspace*{0.3cm}
 \begin{adjustbox}{addcode={
\begin{minipage}{\width}}
{\caption{Bifurcation diagram assuming that $\Gamma_0$ is stable.}
\label{stable}
\end{minipage}}, rotate=90,center}
\begin{overpic}[width=21cm]{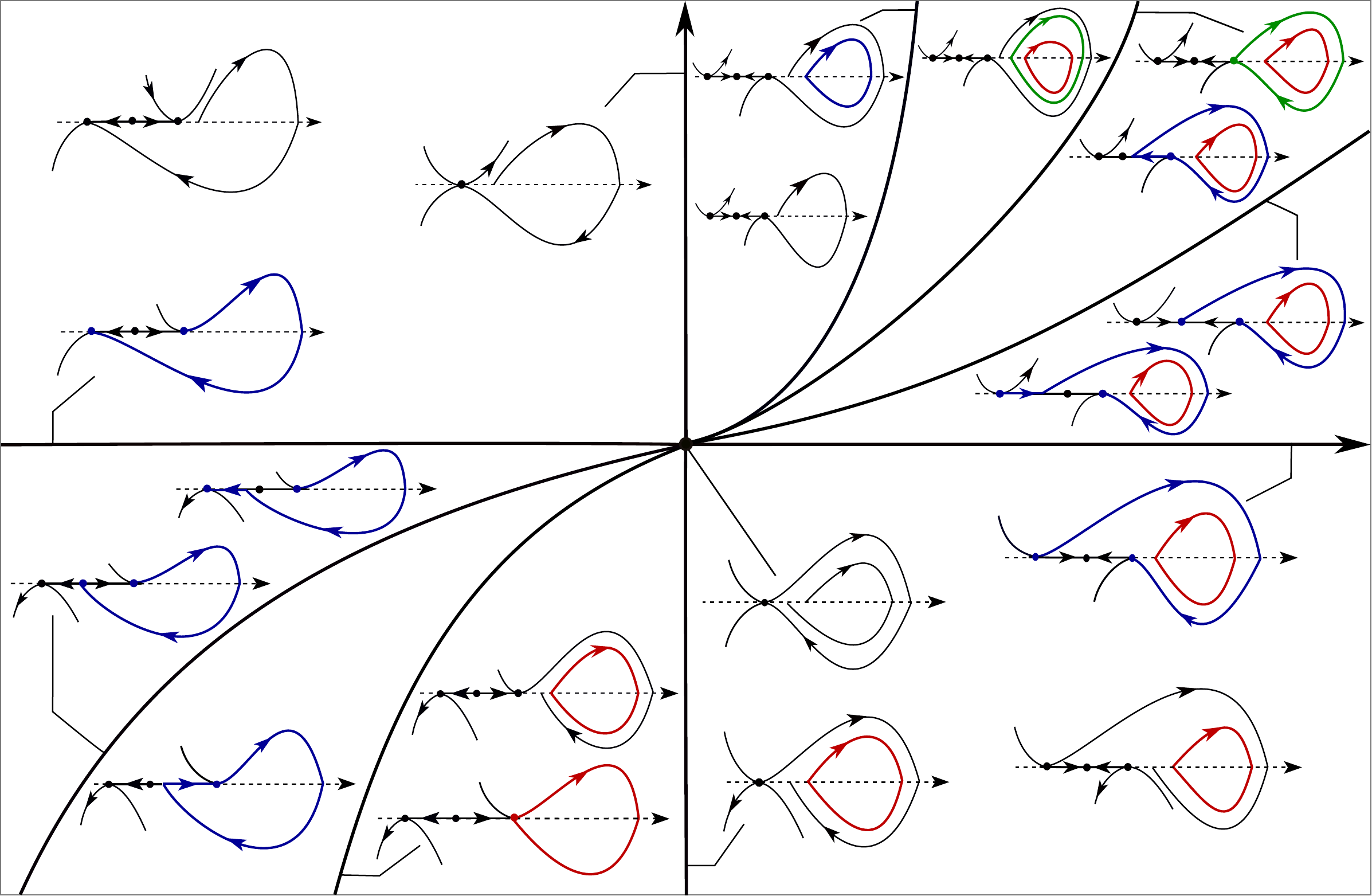}
			\begin{footnotesize}
                \put(98,31){$\alpha$}
				\put(51,63){$\beta$}
				\put(55.3,22){$0$}
				
				\put(63,10){$\gamma_1$}
				\put(54.9,8.8){$0$}
					
                \put(88.7,10.7){$\gamma_1$}
                \put(75.8,10){$0$}
                \put(81.5,8.1){$\alpha$}  
                \put(79.3,10.1){$p^s$}             
                
                \put(87.5,26.5){$\gamma_1$}
                \put(75.2,25.2){$0$}
                \put(82,23){$\alpha$}                 
                \put(79,25.5){$p^s$}           
                
                \put(85,37.3){$\gamma_1$}
                \put(72.5,37.3){$0$}
                \put(80,35){$\alpha$}
                \put(77,35){$p^s$}                
                
                \put(94.7,43.2){$\gamma_1$}	
                \put(85.7,42.6){$p^s$}
                \put(82.4,42.4){$0$}                
                \put(89.7,40.5){$\alpha$}                
                
                \put(89.3,54.6){$\gamma_1$}
                \put(85,52.5){$\alpha$}
                \put(79.7,54.4){$0$}
                \put(81,52.4){$p^s$}        
                
                \put(94.9,61.4){$\gamma_1$}
                \put(89.5,59.4){$\alpha$}
                \put(84,61.4){$0$}
                \put(87,61.7){$p^s$}               
                
                \put(76,60){$\gamma_1$}               
                \put(76.8,62.7){$\gamma_2$}
                \put(71.7,59.6){$\alpha$}
                \put(67.6,61.5){$0$}
                \put(69.8,61.8){$p^s$}                 
                
                \put(61.3,61){$\gamma_3$}
                \put(55.7,58.3){$\alpha$}
                \put(51.2,60.3){$0$}
                \put(53.4,60.6){$p^s$} 
                
                \put(55.2,48){$\alpha$}
                \put(51.3,50.1){$0$}
                \put(53.6,50.5){$p^s$}  
                
                \put(33.1, 52.4){$0$}
                
                \put(5.8,55){$\alpha$}
                \put(9.1,57.3){$p^u$}
                \put(12.5,57.1){$0$}
                
                \put(6,39.7){$\alpha$}
                \put(9.2,42){$p^u$}
                \put(12.8,41.6){$0$}
                
                \put(14.5,28.4){$\alpha$}
                \put(18.1,30.3){$p^u$}               
                \put(21.2,30.3){$0$}
                
                \put(9.4,23.4){$0$}
                \put(5.7,23.6){$p^u$}
                \put(2.7,21.1){$\alpha$}
                
                \put(15.4,8.6){$0$}
                \put(10.8,8.9){$p^u$}
                \put(7.5,6.7){$\alpha$}
                
                \put(37,6.4){$0$}
                \put(32.8,6.5){$p^u$}
                \put(29.1,4.3){$\alpha$}
                
                \put(44,16){$\gamma_1$}
                \put(37.3,15.3){$0$}
                \put(34.3,15.7){$p^u$}
                \put(31.7,13.4){$\alpha$}
                
                \put(44,25){$R_1$}
                \put(56,27){$R_2$}
                \put(63,34){$R_3$}
                \put(64,38){$R_4$}
                \put(61,41){$R_5$}
                \put(53,43){$R_6$}
                \put(40,41){$R_7$}
                \put(36,31){$R_8$}
                \put(38,28){$R_9$}
                
                \put(49,-2){$B_{12}$}
                \put(100.5,32){$B_{23}$}
                \put(100.5,55){$B_{34}$}
                \put(82,66){$B_{45}$}
                \put(66,66){$B_{56}$}
                \put(49,66){$B_{67}$}
                \put(-3,32){$B_{78}$}
                \put(0.5,-2){$B_{89}$}
                \put(23,-2){$B_{91}$}
			\end{footnotesize}
		\end{overpic}
 \end{adjustbox}
\end{figure}

\begin{figure}
		\vspace*{0.3cm}
  \begin{adjustbox}{addcode={
\begin{minipage}{\width}}
{\caption{Bifurcation diagram assuming that $\Gamma_0$ is unstable.}
\label{repulsive}
\end{minipage}}, rotate=90,center}    
\begin{overpic}[width=21cm]{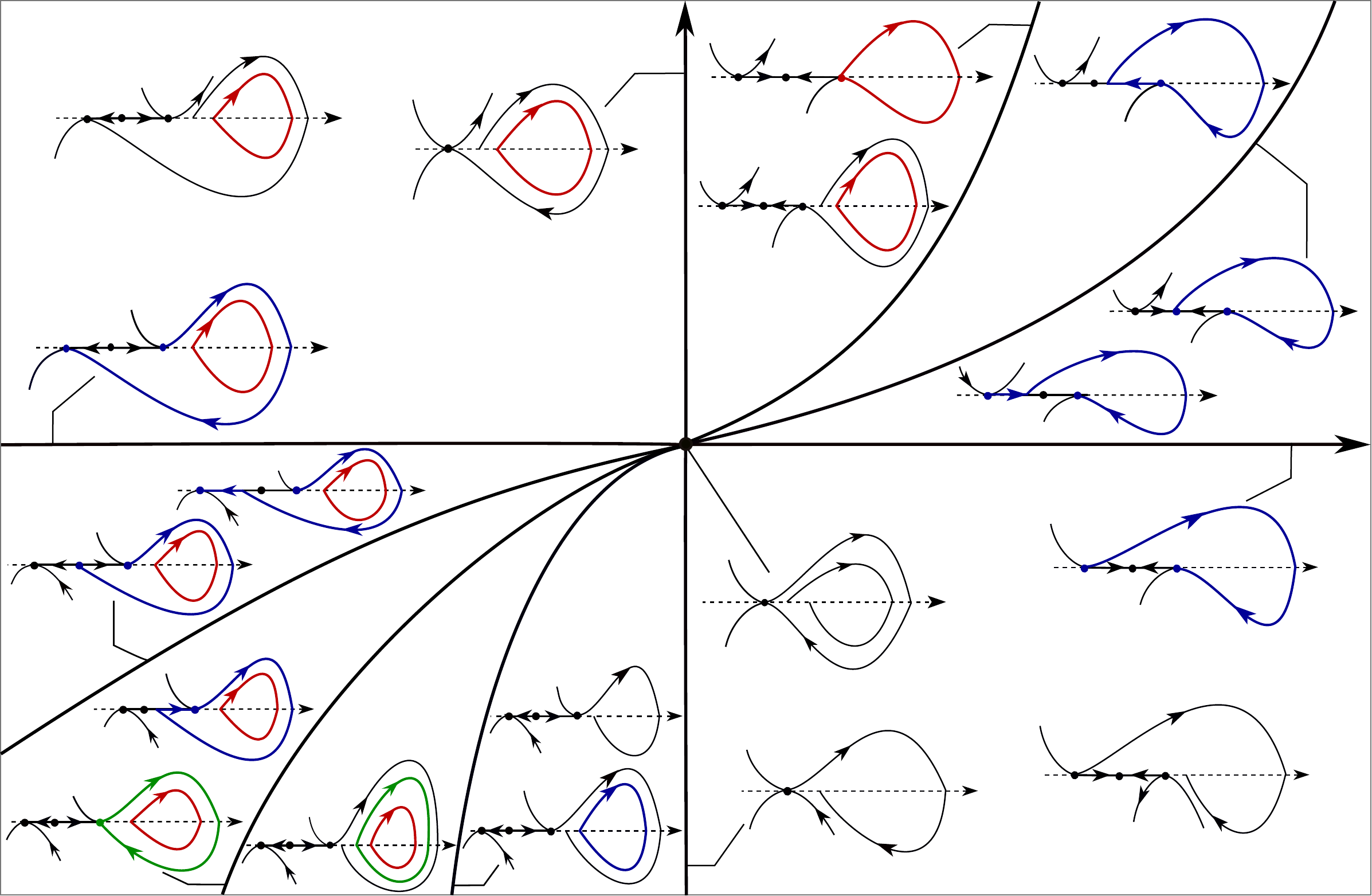}
			\begin{footnotesize}
				\put(98,31){$\alpha$}
				\put(51,63){$\beta$}
				\put(55.3,22){$0$}
				
				\put(57,8.2){$0$}  						
				
				\put(77.9,9.3){$0$}
				\put(81.1,7.2){$p^s$}
				\put(84.5,7.3){$\alpha$}
				
		     	\put(78.7,24.4){$0$}
		     	\put(81.8,22.2){$p^s$}
				\put(85,22.3){$\alpha$}				
				
				\put(71.6,37){$0$}
				\put(74.5,34.9){$p^s$}
				\put(78,35){$\alpha$}
				
				\put(82.4,43.2){$0$}
				\put(85,40.9){$p^s$}
				\put(89,41){$\alpha$}
				
				\put(77,59.8){$0$}
				\put(79.5,57.5){$p^s$}
				\put(84,57.7){$\alpha$}
				
				\put(53.4,60.3){$0$}
				\put(57,60.7){$p^s$}             
				\put(61,58.2){$\alpha$}						
				
				\put(64,52){$\gamma_1$} 
				\put(52.2,50.8){$0$}
				\put(55.4,51.2){$p^s$}             
				\put(57.8,48.6){$\alpha$}				
				
				\put(40,56.5){$\gamma_1$}
				\put(32.2,55){$0$} 				
				
				\put(19,58){$\gamma_1$}
				\put(11.8,57.2){$0$}
				\put(8.3,57.6){$p^u$}             
				\put(5.8,55){$\alpha$}				
				
				\put(17.5,40.9){$\gamma_1$}
				\put(11.4,40.5){$0$}
				\put(7.7,40.7){$p^u$}
				\put(4.3,38.6){$\alpha$}				
				
				\put(26,30.5){$\gamma_1$}
				\put(21.1,30.2){$0$}
				\put(18.2,30.2){$p^u$}
				\put(13.8,28){$\alpha$}
			    
			    \put(13.8,24.7){$\gamma_1$}
			    \put(2,22.4){$\alpha$}
			    \put(5.3,24.9){$p^u$}
			    \put(8.8,24.6){$0$}	  
			    
			    \put(18.2,14.6){$\gamma_1$}
			    \put(8.3,12){$\alpha$}
			    \put(10.1,14.3){$p^u$}
			    \put(13.8,14.3){$0$}
			    
				\put(12,6.2){$\gamma_1$}	
				\put(6.8,6){$0$}
				\put(3.4,6.4){$p^u$}
				\put(1.4,4){$\alpha$}
				
				\put(28.4,4.6){$\gamma_1$}
                \put(28.9,6.9){$\gamma_2$}
                \put(23.7,2.1){$0$}
                \put(20.8,4.6){$p^u$}
                \put(18.5,2.1){$\alpha$}
                
                \put(44.9,6.2){$\gamma_3$}
                \put(39.8,5.2){$0$}
                \put(36.7,5.6){$p^u$}
                \put(34.5,3){$\alpha$}
                                    
                \put(41.3,11.5){$0$}
                \put(38.7,14){$p^u$}
                \put(36.4,11.4){$\alpha$}                    
			\end{footnotesize}
		\end{overpic}
  \end{adjustbox}
\end{figure}
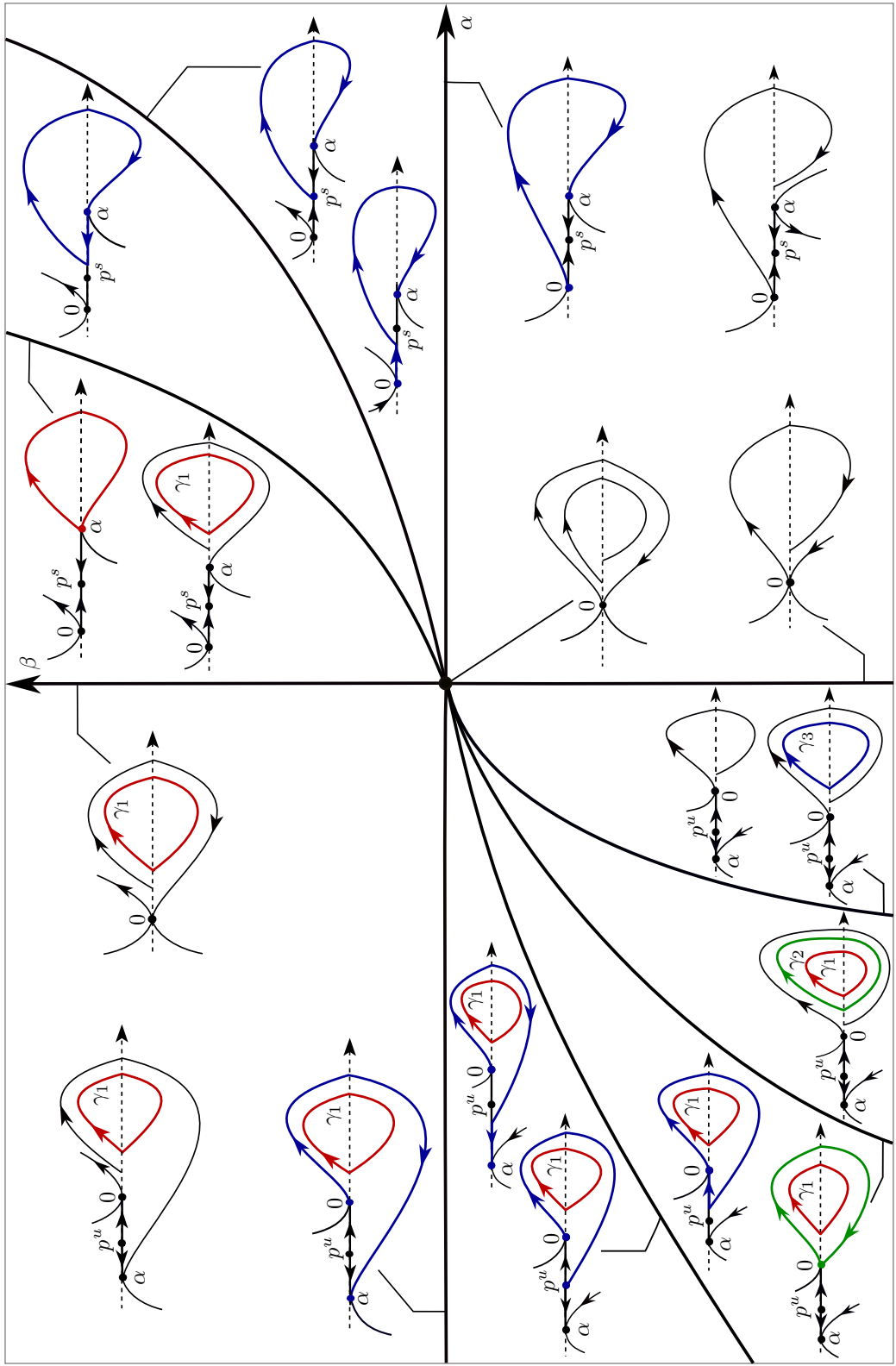
\newpage

\section{Preliminary results}\label{proofs}

This section is devoted to provide some preliminary results needed for proving Theorems  \ref{R2} and \ref{T:BD}. In Subsection \ref{subdisp} we construct the displacement function $f_Z$ of $Z$. We shall see that  the crossing and critical crossing cycles of $Z$ correspond  to the zeros of $f_Z$, so in Subsection \ref{subzeros} we study these zeros.  Finally, in Subsection \ref{sliding} we analzse the sliding dynamics.

\subsection{Displacement function}\label{subdisp}

In what follows the neighborhoods $\mathcal{A}_0$ and $\V_0$ will be reduced if necessary.  Moreover, for $Z\in\V_0$ we denote $\eta(Z)=(\alpha,\beta)$, where  $\eta$ is defined in \eqref{bifun}.

For $\xi>0$ sufficiently small, the section  $\Sigma^+=\{ (x,\xi):~(x,0)\in\Sigma \}$ is transversal to the trajectory of $X_0$ passing through $(0,0)$ forward in time. By transversality, for each $Z=(X,Y)\in\V_0$,  $\Sigma^+$ is also transversal to the trajectory of $X$ starting at $(0,0)$ forward in time. 
Denote by $(p^+,\xi)$ the first intersection of this trajectory with $\Sigma^+$. From the Tubular Flow Theorem, there exists  $\delta>0$ such that the trajectory starting at  $(x,0)$, for $x\in[0,\delta)$, intersects $\Sigma^+$ transversally at $(\rho_1(x),\xi)$ nearby $p^+$. Accordingly, we have well defined a {\it transition  map } $\rho_1: [0,\delta) \longmapsto \Sigma^+$ (see Figure \ref{transition}).

\begin{lemma}
\label{lemma1}
The transition map $\rho_1: [0,\delta) \longmapsto \Sigma^+$ around $x=0$ is given by 
\begin{equation}
\label{rho1}
\rho_1(x) =p^+ + {\ell_1}(Z) x^2 + \CO(x^3),
\end{equation}
with ${\ell_1}(Z) > 0$.
\end{lemma}
\begin{proof}
Consider a normal form $\widetilde X$ of $X$ nearby $(0,0)$, for instance $\widetilde X(x,y)=(1,x)$ (see \cite{Vis}). Denote   the trajectory of $\widetilde X$ passing thought $(x,0) \in \Sigma$  by $\dis\varphi(t,x,0) = (\varphi_1(t,x,0),\varphi_2(t,x,0))$. It is easy to see that $T(x)=-x + \sqrt{x^2+ 2  \xi}$ is the first positive time spent by $\dis\varphi(t,x,0)$ to intercept $\Sigma^+$. Then the first component of the intersection between the trajectory $\dis\varphi(t,x,0)$ and the section $\Sigma^+$ can be expanded in Taylor series around $x=0$ as
$$
\varphi_1(T,x,0)=a + bx^2 + \CO(x^3),~\text{where}~ a=\sqrt{2\xi}~\text{and}~b =1/(2\sqrt{2\xi})>0.
$$
Since $X$ is locally conjugated to $\widetilde X$, the trajectory of $X$ passing thought $(x,0)$ is given by $\psi(t,x,0) =(\psi_1(t,x,0),\psi_2(t,x,0))  = h^{-1}\circ \varphi (t,h(x,0))$, being $h:\R^2  \to \R^2$  a diffeomorphism such that $h(0)=0$.
Therefore, $\rho^1_Z(x)=\dis\psi_1(T^*(x) ,x,0)$ can be expanded in Taylor series around $x=0$ as $\eqref{rho1}$, where $T^*(x)$ is the positive time  spent by $\psi(t,x,0)$ to intercept $\Sigma^+$. 
\end{proof}

Analogously, the section $\Sigma^-= \{ (x,-\xi):~(x,0)\in \Sigma \}$ is transversal to the trajectory of $Y_0$ passing through $(\alpha,0)$ backward in time. Then $\Sigma^-$ is transversal to the trajectory of $Y$ starting at $(\alpha,0)$ backward in time. Denote by  $(p^-,-\xi)$ the first intersection of this trajectory with $\Sigma^-$. Then there exists $\wtilde{\delta}>0$ such that the trajectory starting at  $(x,0)$ backward in time, for $x\in[\alpha,\alpha+\wtilde{\delta})$, intersects $\Sigma^-$ transversally at $(\mu_1(x),-\xi)$ nearby $p^-$. Therefore,
we have well defined a {\it transition  map } $\mu_1: [\alpha,\alpha+\wtilde{\delta}) \longmapsto \Sigma^-$ (see Figure \ref{transition}), and this map is characterized in the following lemma.

\begin{lemma}\label{lemma2}
The transition  map $\mu_1: [\alpha, \alpha + \wtilde{\delta}) \to \Sigma^-$ around $x=\alpha$ is given by
\begin{equation*}
\label{mu1}
\mu_Z^1(x)= p^- + k_1(Z)( x - \alpha)^2 + \CO((x - \alpha)^3)
\end{equation*}
with $k_1(Z)  > 0$. 
\end{lemma}
\begin{proof}
This proof follows analogously to the proof of Lemma \ref{lemma1}, but now considering the normal form  $\widetilde Y(x,y)=(-1,x-\alpha)$. 
\end{proof}

\begin{figure}[h]
	\begin{overpic}[width=7cm]{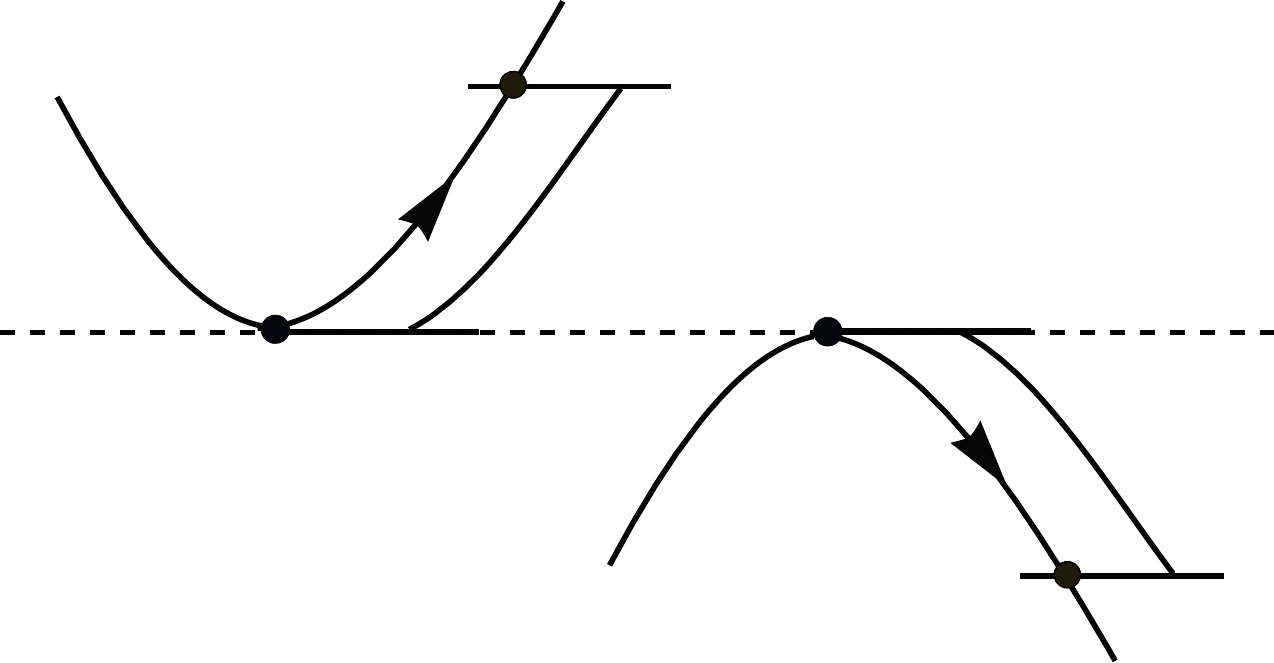}
		\put(20,20){$0$}
		\put(63,20){$\alpha$}
		\put(89,16){$\mu_1$}
		\put(44,34){$\rho_1$}
		\put(2,28){$\Sigma$} 
		\begin{small} 
			\put(55,43){$\Sigma^+$}
			\put(24,48){$(p^+,\xi)$}
			\put(97.5,5){$\Sigma^-$}
			\put(63,0){$(p^-,-\xi)$}
		\end{small}
	\end{overpic}
	\caption{Transitions maps for the vector field $Z$.}
	\label{transition}
\end{figure}
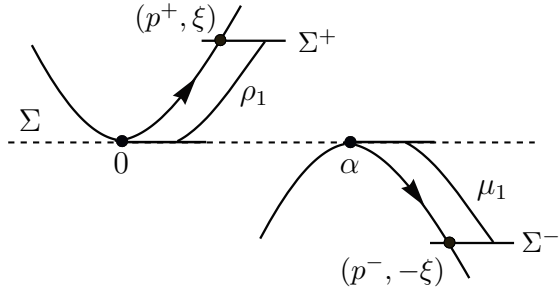

Since the trajectory of $X$  reaches $\Sigma$ transversally forward in time, it also defines a diffeomorphism $\rho_2: \rho_1([0,\delta)) \to \Sigma$, which can be expanded around $p^+$ as
\begin{equation}
\label{rho2}
\rho_2(x)=q_{X} + \ell_2(Z) (x-p^+) + \CO((x -p^+)^2),
\end{equation}
with $\ell_2(Z) <0$. Similarly, the trajectory of $Y$ backward in time defines a  diffeomorphism  $\mu_2:  \mu_1([\alpha,\alpha+\wtilde{\alpha}))\to  \Sigma$,  which can be expanded around $p^-$ as
\begin{equation}
\label{mu2}
\mu_2(x)= q_Y + k_2 (Z) (x - p^-) + \CO((x- p^-)^2),
\end{equation}
with $k_2 (Z) <0$  (see Figure \ref{dis1}).

\begin{figure}[h] 
	\begin{center}
		\begin{overpic}[width=10cm]{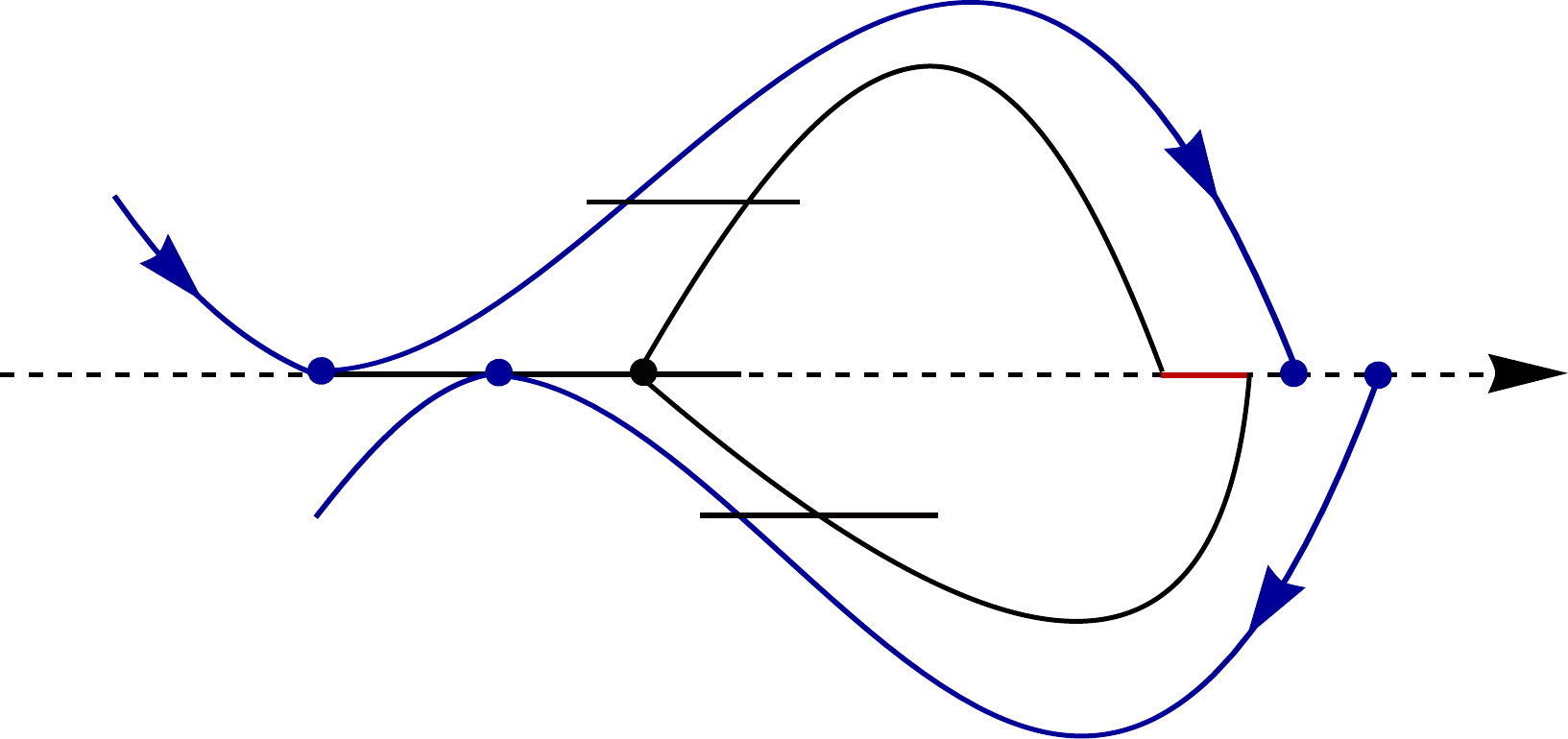}
				\put(1,25){$\Sigma$} 
				\put(19.3,25){$0$}
				\put(31,20){$\alpha$}
				\put(83,26){$q_{X}$}
				\put(88,20){$q_{Y}$}
				\put(45,27){${\rho}_1$}
				\put(57,39){${\rho}_2$}
				\put(67,10){${\mu}_2$}
				\put(49,18){${\mu}_1$}
		\end{overpic}
	\end{center}
	\caption{Displacement function $f_Z$ associated with $Z \in \V_0$.}
	\label{dis1}
\end{figure}
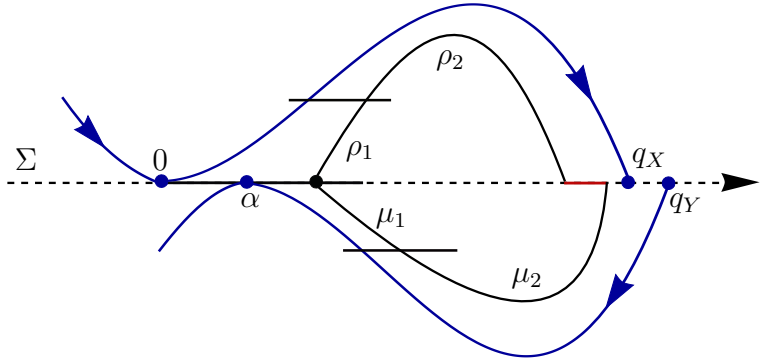

Now define $\sigma(\alpha) = [ 0, \delta \big) \cap [ \alpha, \alpha +\wtilde\delta \big)$. Note that 
\begin{equation}
\label{I}
\sigma(\alpha)=[A_\alpha, \lambda), ~ \text{for a fixed} ~ \lambda> 0, ~\text{where} ~A_\alpha=\max\{ 0,\alpha\} .
\end{equation}
Hence the  displacement function $f_Z: \sigma(\alpha) \to \R$ associated with $Z$ is given by
\begin{equation*}
f_Z(x) = \rho_2 \circ \rho_1 (x) - \mu_2 \circ \mu_1(x).
\end{equation*} 

Note that  $\mathcal{A}_0 \cap \Sigma= \mathcal{A}_{01} \dot{\cup}\mathcal{A}_{02}$, where $\mathcal{A}_{01}$ and   $\mathcal{A}_{02}$ are open connected components in   $\Sigma$ with $(0,0) \in \mathcal{A}_{01}$.  Without loss of generality the annulus $\mathcal{A}_0$ can be taken in order that $\sigma(\alpha)= \{ (x,0) \in \mathcal{A}_{01}: ~ x \geq A_\alpha \}$. The next proposition is obtained directly from Lemmas \ref{lemma1} and \ref{lemma2}, and  expressions \eqref{rho2} and \eqref{mu2}.

\begin{proposition}\label{itemi}
Assume that $\Gamma_0$ is a simple two-fold cycle of $Z_0\in\Omega$. Then there exist an annulus $\mathcal{A}_0$ of $\Gamma_0$ and a neighborhood $\V_0\subset \Omega$ of $Z_0$  such that, for each $Z\in\V_0$, the displacement function $f_{Z}:\sigma(\al)\rightarrow\Sigma$ reads
\begin{equation*}
\label{fz}
f_Z(x) =\beta + \ell (Z)x^2 - k (Z)(x-\alpha)^2 + \CO(x^3) + \CO((x-\alpha)^3),
\end{equation*}
with $\ell (Z)= \ell_1 (Z)\ell_2 (Z) <0$  and  $k (Z)=k_1(Z) k_2(Z)<0$.  In particular, for $Z=Z_0$, we have
$f_{Z_0}(x) = M x^2 + \CO(x^3)$, with $M=\ell(Z_0) - k(Z_0) \in \R$.
\end{proposition}

Note that the zeros of $f_Z$ correspond to the crossing and critical crossing cycles of $Z$ nearby $\Gamma$. 

\begin{remark}\label{defi}
If $\x =A_\alpha$ and $f_Z(\x)=0$, then there exists a {\it critical crossing cycle} (see \cite{FPT}) passing through $(A_\alpha,0)$.  If $\x \in \inte \sigma (\alpha)$ and $f_Z(\x)=0$, then there exists a {\it crossing cycle} passing through $(\x,0)$. Furthermore, if $f'_Z(\x) >0$ the cycle is stable, and if $f'_Z(\x) <0$ the cycle unstable. If $f'_Z(\x)=0$ and  $f''_Z(\x) > 0$, the cycle is semi-stable.
\end{remark}

\subsection{Zeros of the displacement function}
\label{subzeros}

Assume that $M \neq 0$. If  $M >0$ (resp. $M <0$), then  $f_{Z_0}(x) < 0$  (resp. $f_{Z_0}(x) > 0$) for $x\in\sigma(0)$, consequently $\Gamma$ is a stable  (resp. unstable) simple two-fold cycle.  For the sake of simplicity we denote $\ell=\ell(Z)$, $k=k(Z)$, and $L=\ell - k$. Note that $L\neq 0$ for every $Z \in \V_0$ and $L \to M$ when $(\alpha,\beta) \to (0,0)$.

In order to study the existence of cycles of $Z$ nearby $\Gamma$  we shall study the zeros of $f_Z$ contained in $\sigma(\alpha$) when $(\alpha,\beta)$ is taken in a small neighborhood of the origin.

\begin{lemma} \label{lemma3} Let $U=\eta(\V_0)$. Then there is a unique $C^r$ curve $\beta=\beta_1(\al)$ defined in a small neighborhood of $0\in\R$ for which $f_Z$ has a zero $x(\al)$ of multiplicity two for $(\alpha,\beta_1(\alpha))\in U_1$. Moreover, 
\begin{equation*}
x(\alpha)= -\frac{k}{L} \alpha + \CO(\alpha^2) \quad \text{and} \quad \beta_1 (\alpha)= \frac{k \ell}{L} \alpha^2 + \CO(\alpha^3).
\end{equation*}
\end{lemma}
\begin{proof}

Taking $y= x-\alpha$ we verify that the zeros of $f_Z$ are in one-to-one correspon-dence with the solutions of the system:
\begin{equation}
\label{ps-fz}
\left\{\begin{array}{l}
0= F_\beta(x,y)=\beta + \ell x^2 - k y^2 + \CO(x^3) + \CO(y^3),\\
0=G_\alpha(x,y)= \alpha+y-x.
\end{array}\right.
\end{equation}
If $(\hat x, \hat y)$ is a solution of \eqref{ps-fz} for some $(\alpha,\beta)\in\R^2$, then the level curves $F_\beta=0$ and $G_\alpha=0$ intersect each other at $(\hat x, \hat y)$.  In this context, a zero $\hat x$ of $f_Z$ of multiplicity two corresponds to a quadratic contact at $(\hat x,\hat x-\al)$ between the curves $F_\beta=0$ and $G_\alpha=0$.  

The contact points between the level curves of $F_\beta$ and $G_\alpha$ are given by the zeros of $C(x,y)$, where 
\begin{equation*}
C(x,y)=  \nabla F_\beta (x,y) \cdot {\nabla G_\alpha}(x,y)^\perp = 2 \ell x - 2 ky +\CO(x^2) + \CO(y^2).
\end{equation*}
Since  $C(0,0)=0$ and $\dis\frac{\p C}{\p y}(0,0) = - 2 k \neq 0$, it follows from the Implicit Function Theorem that there exists a unique $C^r$ function $y(x)$, defined in a neighborhood of $0\in\R$, with $y(0)=0$ and $y'(0)= \dis\frac{\ell}{k}\neq 1$, such that $C(x,y(x))=0$. In this case,
$$
y(x) = \dis\frac{\ell}{k}  x + \CO(x^2). 
$$
Moreover,
\[
\nabla C (x,y(x))  \cdot {\nabla G_\alpha} (x,y(x)) ^\perp = 2L + \CO(x) \neq  0.
\]
Therefore, $(x,y(x))$ is a quadratic contact between the level curves of $F_\beta$ and $G_\alpha$. 

Now we have to find parameters $(\al,\beta)$ in order that this contact happen at $F_\beta=0$ and $G_\alpha=0$. So denote
$$
\widetilde G(x,\alpha)= G_\alpha(x,y(x))=\alpha + \bigg(\frac{\ell}{k} -1\bigg) x + \CO(x^2),
$$

Since $\widetilde G(0,0)=0$ and  $\dfrac{\p \widetilde G}{\p x}(0,0) = \gamma \neq 0$, it follows from the Implicit Function Theorem  that there exists a unique $C^r$ function $x(\alpha)$, defined in a small neighborhood of $0\in\R$, with $x(0)=0$ and $x'(0)=-\dfrac{k}{L}\neq0$, such that  $\widetilde G(x(\alpha),\alpha))=0$. In this case,
\begin{equation*}
x(\alpha)= -\frac{k}{L} \alpha + \CO(\alpha^2).
\end{equation*}

Finally, we may find a $C^r$ curve $\beta=\beta_1(\al)$ such that $F_{\beta_1(\al)}(x(\alpha),y(x(\alpha)))=0$. Indeed,
$$
F_\beta(x(\alpha),y(x(\alpha)))=\beta - \frac{k  \ell}{L} \alpha^2 + \CO(\alpha^3),
$$
so we conclude that
\begin{equation*}
\beta_1 (\alpha)= \frac{k \ell}{L} \alpha^2 + \CO(\alpha^3). 
\end{equation*}
The uniqueness of $\beta_1(\al)$ is assured again by the Implicit Function Theorem.

Hence if $\beta=\beta_1(\al)$, then $x(\alpha)$ is a zero of multiplicity two of $f_Z$ for $(\al,\beta_1(\alpha))$ in a neighborhood of $(0,0)$ contained in $U=\eta(\V_0)$. 
\end{proof}

From Lemma \ref{lemma3}, the level curves $F_\beta=0$ and $G_\alpha=0$ intersect each other transversally when $\beta \neq \beta_1(\alpha)$.  
In the next proposition, in order to obtain all the zeros of $f_Z$, we study what happens for $\beta>\beta_1(\al)$ and for $\beta<\beta_1(\al)$. This behavior changes with the sign of $M$. From now on we shall assume that $M >0$. The analysis for $M < 0$ will be similar.

\begin{lemma}\label{lemzeros} Let $V$ be the neighborhood given in Lemma \ref{lemma3}. There exists a neighborhood $U_1\subset V$ of $(0,0)$ such that the following statements hold for $(\al,\beta)\in U_1$:
\begin{itemize}
\item [$(a)$] If $\beta > \beta_1(\alpha)$, then $f_Z$ has no zeros in $U_1$.
\item [$(b)$] If $\beta = \beta_1(\alpha)$, then $f_Z$ has a zero $x^*$ of multiplicity two with $f'_Z(x^*)=0$ and $f''_Z(x^*)>0$.
\item [$(c)$] If $\beta < \beta_1(\alpha)$, then $f_Z$ has two distinct zeros, $x^-$ and $x^+$, with $f'_Z(x^+)>0$ and $f'_Z(x^-)<0$. 
\end{itemize} 
\end{lemma}
\begin{proof}
Consider the functions $F_\beta$ and $G_\alpha$ defined in \eqref{ps-fz}. Given $\e_0>0$, there exists a neighborhood $I\subset \R$ of $0$ such that 
\[ 
K=\{(\alpha,\beta)\in\R^2:\,\al\in\overline{I},\,\beta=\beta_1(\al)+\e,\,\e\in[-\e_0,\e_0]\}\subset V.
\]
Note that for each  $(\alpha,\beta)\in K$ we have
$$
f_Z(x)=f_\e(x) = \e + F_{\beta_1(\alpha)}(x,x-\alpha).
$$

From Lemma \ref{lemma3}, $x(\alpha)$ is a non-degenerate singularity of the map $x\mapsto f_\e(x) -\e$. Indeed,
\[
f_\e(x(\alpha)) =\e, ~ \dfrac{\p f_\e}{\p x} (x(\alpha))=0,~ \text{and} ~\dfrac{\p^2 f_\e}{\p x^2} (x(\alpha))= 2 L + \CO(\alpha) > 0,
\]
From Morse Theory (see \cite{mat}) one may find a local diffeomorphism $\varphi$ around $x(\alpha)$ such that $\f(x(\al))=0$, $\f'(x(\al))=1$, and $f_\e \circ \varphi^{-1} (v) =\e+A v^2$. In this case,  
 $\varphi^{-1}(v)=x(\alpha)+v + \CO(v^2)$ and
 \[
 A=\dis \frac{1}{2}\frac{\p^2 f_\e}{\p x^2} (x(\alpha))= L +\CO(\alpha)>0.
 \]
Hence we have the following possibilities: 
\begin{itemize}
\item[$(a)$] If $\e>0$, then $f_\e \circ \varphi^{-1}$ has no zeros.
\item[$(b)$]  If $\e<0$, then $f_\e \circ \varphi^{-1}$ has two distinct zeros, namely
$$
v^\pm = \pm \sqrt{-\frac{\e}{A}}=\pm\sqrt{-\frac{\e}{L}}+\CO(\al\sqrt{-\e}).
$$
\item[$(c)$] If $\e=0$, then $v^*=0$ is a zero of multiplicity two of $f_\e \circ \varphi^{-1}$. 
\end{itemize}
These zeros are in one-to-one correspondence with the zeros of $f_\e$ by the diffeomorphism $\varphi$. Indeed,  
\begin{equation}
\label{zeros}
\begin{array}{l}
x^*=\varphi^{-1}(v^*)=x(\alpha)=-\dfrac{k}{L} \alpha  +\CO(\alpha^2) ~ \text{for}~ \e=0,~\text{and}\vspace{0.2cm}\\

x^\pm=\varphi^{-1}(v^\pm)=-\dfrac{k}{L} \alpha \pm \dis \sqrt{-\frac{\e}{L}} +\CO_2 \left( \alpha,\sqrt{-\e}\right)~\text{for}~\e<0.
\end{array}
\end{equation}
The term $\CO_2 \left( \alpha,\sqrt{-\e}\right)$ is an abbreviation for $\CO(\al \sqrt{-\e})+\CO(\al^2)+\CO(\e)$.

In addition $(f_\e \circ \varphi^{-1})'(v^+) >0$, $(f_\e \circ \varphi^{-1})'(v^-) <0$, and $(f_0 \circ \varphi^{-1})'(v^*) =0$ with $(f_0 \circ \varphi^{-1})''(v^*) >0$. Since $(\varphi^{-1})'(v)>0$ we also have  $f_\e'(x^+) >0$, $f_\e'(x^-) <0$, and $f_0'(x^*) =0$ with $f_0''(x^*) >0$. We conclude the proof by taking $U_1=\inte K$ (see Figure \ref{U}).

\begin{figure}[h] 
 	\begin{center}
 	\begin{overpic}[width=5cm]{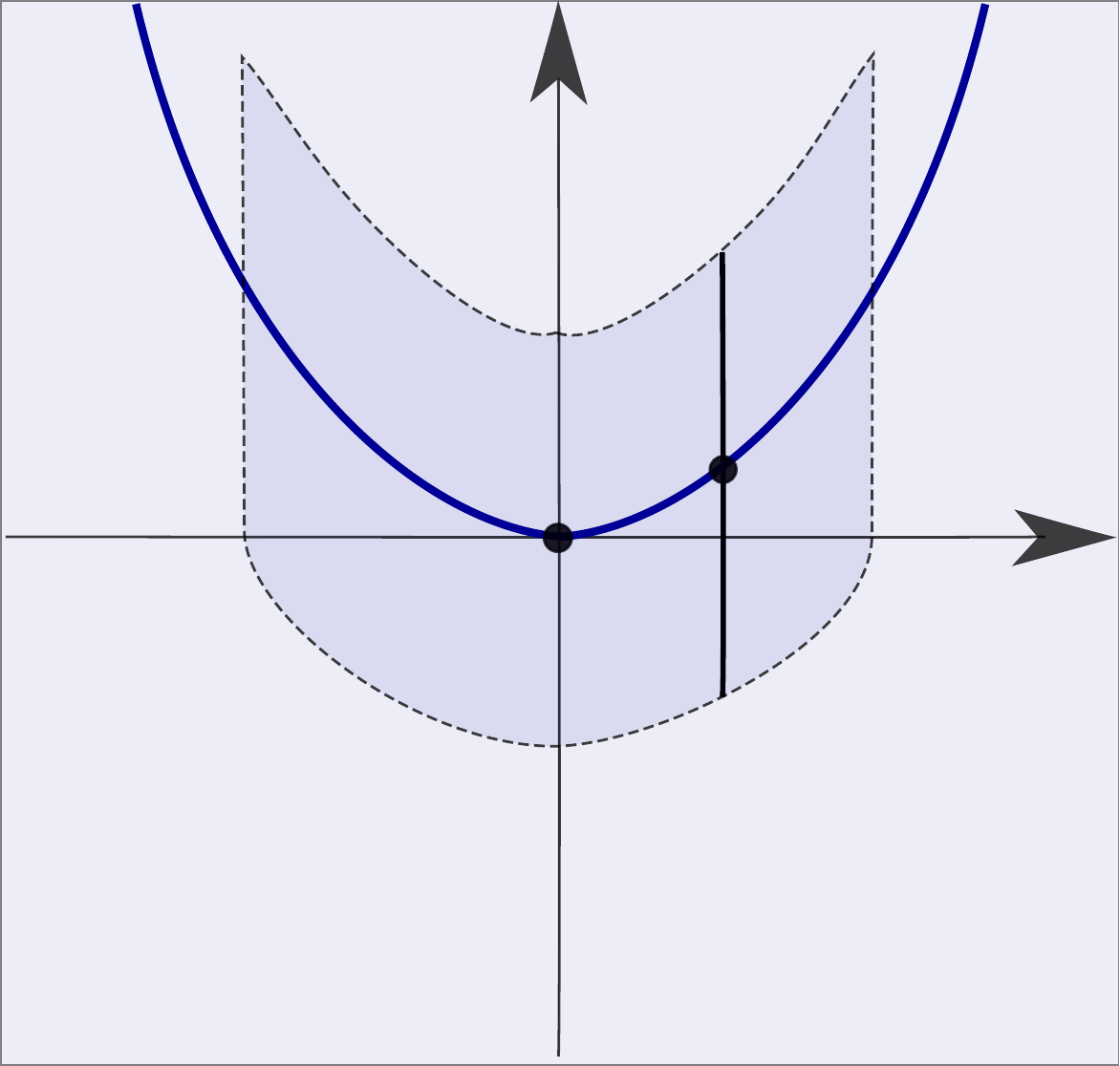}
 			\begin{footnotesize}
 				\put(94,41){$\alpha$} 
 				\put(53,89){$\beta$} 
				\put(5,5){$V$} 
                \put(25,70){$U_1$}
 				\put(63,27){$\beta_1(\alpha)+\e$} 
 			\end{footnotesize}			
 		\end{overpic}
 	\end{center}
 	\caption{Neighborhood $U_1\subset  \R{^2}$  of $(0,0)$.}
 	\label{U}
 \end{figure}
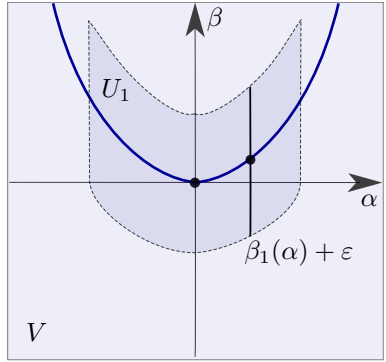
 \end{proof}

In the next lemma we shall determine for what values of $(\alpha,\beta)\in U_1$ the prescribed zeros \eqref{zeros} of $f_Z$ are contained in $\sigma (\alpha)$ (see \eqref{I}). The neighborhood $U_1$ will be reduced if necessary.
\begin{lemma} 
\label{lemzeros2} 
Consider the zeros of $f_Z$ given in \eqref{zeros}, namely $x^*$ for $\beta=\beta_1(\al)$, and $x^{\pm}$ for $\beta < \beta_1(\alpha)$. Then there exist two curves $\beta=\beta_2(\al)$ and $\beta=\beta_3(\al)$, defined in a small neighborhood of $0\in\R$, such that the following statements hold for $(\alpha,\beta) \in U_1$:
\begin{itemize}
\item [$(a)$]  Let $\beta=\beta_1(\al)$. If $\al<0$, then $x^*\notin \sigma (\alpha)$; if $\alpha=0$, then $x^*=A_\al$; if $\alpha>0$, then $x^* \in \inte \sigma (\alpha)$.

\item [$(b)$] Let $\beta<\beta_1(\al)$. If $\alpha \leq 0$, then $x^- \notin \sigma (\alpha)$.  If $\alpha >0$, then $x^- \notin  \sigma (\alpha)$ for $\beta < \beta_2(\alpha)$; $x^- =A_\al$ for $\beta=\beta_2(\alpha)$; and $x^- \in \inte \sigma (\alpha)$ for $\beta > \beta_2(\alpha)$.

\item [$(c)$] Let $\beta<\beta_1(\al)$.  If $\alpha \geq 0$, then $x^+ \in \inte \sigma (\alpha)$.  If $\alpha <0$, then $x^+ \in \inte \sigma (\alpha)$ for $\beta < \beta_3(\alpha)$; $x^+ =A_\al$ for  $\beta=\beta_3(\alpha)$; and $x^+ \notin  \sigma (\alpha)$ for $\beta > \beta_3(\alpha)$.
\end{itemize}
Moreover,
\begin{equation*}
\beta_2(\alpha)=- \ell \alpha^2+\CO(\alpha^3) \quad \text{and} \quad \beta_3(\alpha)=k \alpha^2+\CO(\alpha^3),
\end{equation*}
\end{lemma}
\begin{proof}

Statement $(a)$ follows directly from Lemma \ref{lemma3}. Indeed, 
\[
x^*=x(\alpha)=-\dfrac{k}{L}\al+\CO(\al^2),~\text{with}~ -\frac{k}{L}>1.
\]

To see statement $(b)$, first note that $\al\leq 0$ implies that $x^- < 0$, that is, $x^- \notin \sigma (\alpha)$. Now, assuming $\alpha>0$, we shall determine the values ​of $\alpha$ and $\beta$ for which the equality $x^-=\alpha$ holds.  In \eqref{zeros} we have that $\e<0$. So take $u=\sqrt{-\e}$ and define $h:\R^2 \to \R$ as
$$
h(\alpha,u):=x^-  -\alpha= \left( -\frac{k}{L}  -1\right) \alpha  - \frac{u}{\sqrt{L}}  + \CO_2 (\alpha, u).
$$ 
The right-hand side of the above equality is due to \eqref{zeros}. Since $h(0,0)=0$ and $\dis \dfrac{\p h}{\p u}(0,0)=-\frac{ 1}{\sqrt{L}} $, there exists a unique $C^r$ function $u(\alpha)$, defined in a neighborhood of the origin, with $u(0)  =0$ and $u'(0)=-\dfrac{k}{\sqrt{L}}-\sqrt{L} >0$, such that $h(\alpha,u(\alpha))=0$. In this case,
$$
u(\alpha) = \left(-\dfrac{k}{\sqrt{L}}-\sqrt{L} \right) \alpha + \CO(\alpha^2) >0.
$$
Hence $\dis \e =- u(\alpha)^2= -\left( \frac{k^2}{L} +L +2 k \right)\alpha^2 + \CO(\alpha^3)$. Since $\beta =\beta_1(\alpha)+\e$ we conclude that $x^- =\alpha $ provided that  $
\beta=\beta_2(\alpha)= \beta_1(\alpha) - u(\alpha)^2 =-\ell \alpha^2 + \CO(\alpha^3)$.

For each fixed $\alpha>0$, let $h_{\al}(u)=h(\al,u)$. Notice that $h_{\al}(u(\al))=0$  and $h_{\al}'(u(\al))=\dfrac{\p h}{\p u}(\alpha,u(\alpha)) <0$. Therefore, for each $\al>0$, the function $h_{\al}(u)$ is decreasing in a neighborhood of $u(\al)$. Since $u< u(\alpha)$ if and only if $\beta > \beta_2(\alpha)$, and $u>u(\alpha)$ if and only if $\beta < \beta_2(\alpha)$ we conclude that $x^-\in \inte \sigma (\alpha)$ when $\beta > \beta_2(\alpha)$, and $x^- \notin \sigma (\alpha)$ when $\beta <\beta_2(\alpha)$.  This proves statement $(b)$.

To prove statement $(c)$ we observe that $x^+ > \alpha$, and consequently $x^+ \in \sigma (\alpha)$, provided that $\alpha \geq 0$. When  $\alpha<0$ the proof follows analogously to the proof of Statement (a). In this case, for $u= \sqrt{-\e}$, defining $g(\alpha,u)=x^+$ we get the existence of a function $\beta_3(\alpha)$, defined in a neighborhood of the origin, with $\beta_3(\alpha) = k \alpha^2 +\CO(\alpha^3)$, for which $x^+ =0$ when  $\beta =\beta_3(\alpha)$, $x^+ > 0$ when $\beta < \beta_3(\alpha)$, and $x^+ < 0$ when $\beta > \beta_3(\alpha)$. This proves statement $(c)$.
\end{proof}

\subsection{Sliding dynamics} \label{sliding}

When  $\alpha <0$ (resp. $\alpha > 0$),  then $\Sigma^s = \{(x,0):\, \alpha < x < 0 \}$ (resp. $\Sigma^e = \{(x,0):\, 0  < x < \alpha \}$) is a sliding region (resp. escaping region).  In this case, we observe the existence of a unique pseudo equilibrium of $Z$ nearby $Z_0$ as stated in Lemma \ref{R1} (for instance see Figure \ref{sing}).
\begin{figure}[h]
	\subfigure[ $\alpha <0$]
	{\begin{overpic}[width=4.5cm]{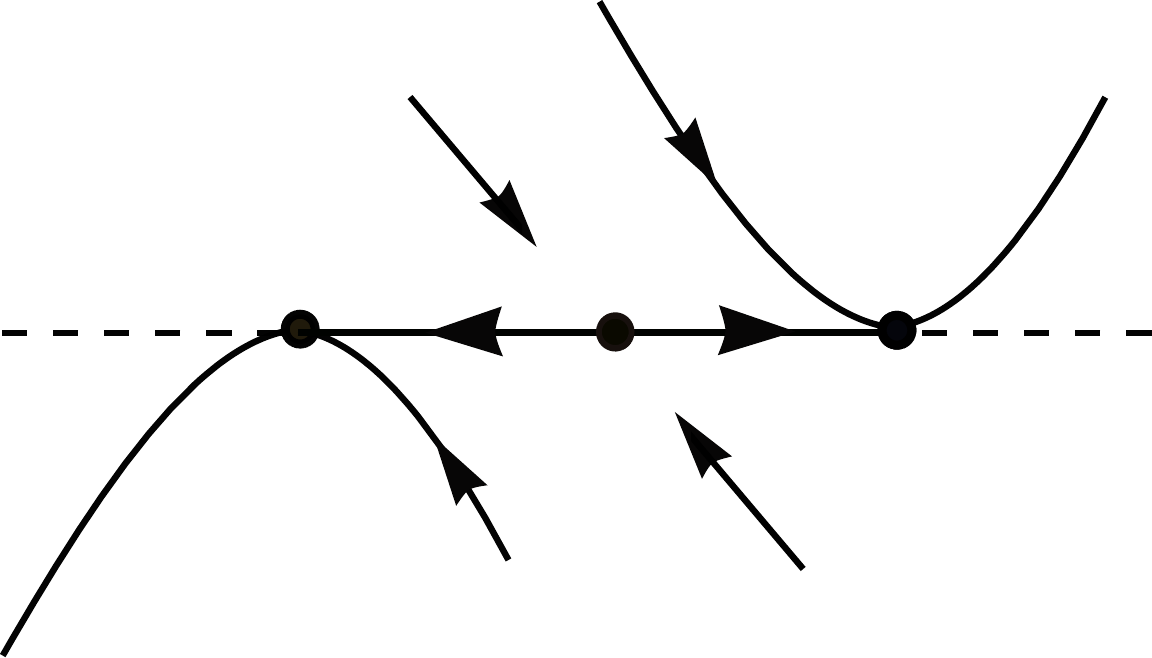}
		\end{overpic}}
		\hspace*{2cm}
		\subfigure[$\alpha >0$]
		{\begin{overpic}[width=4.5cm]{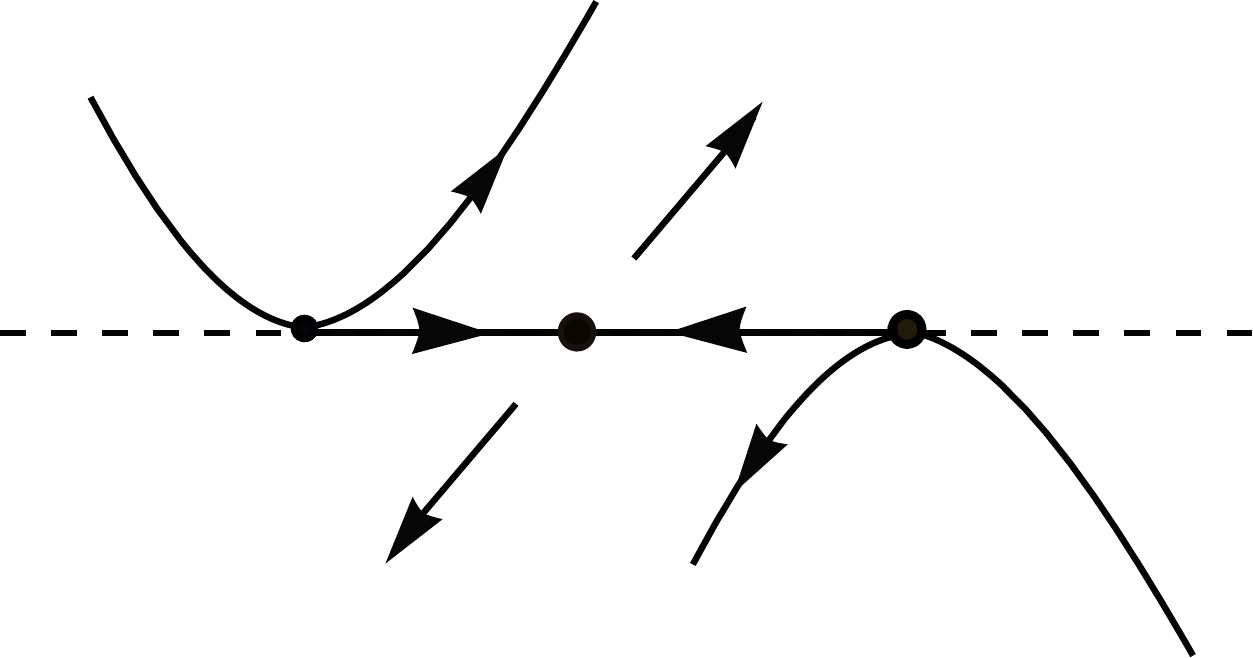}
			\end{overpic}}
			\caption{Sliding and escaping regions, respectively, with their pseudo-equilibria.}
			\label{sing}
		\end{figure}

In the sequence, the neighborhood $\V_0$ will be reduced if necessary.
\begin{lemma}
\label{R1} 
Take $Z=(X,Y) \in \V_0$. If $\alpha < 0$ $($resp. $\alpha > 0)$, then the sliding vector field  defined on $\Sigma^s$ $($resp. the escaping vector field  defined on $\Sigma^e)$ has a unique unstable equilibrium point $(p^u,0)$ $($resp. stable equilibrium point $(p^s,0)$$)$.
\end{lemma}
\begin{proof}			
We prove this lemma for the case $\alpha <0$. Denote by $X^1,X^2$ and $Y^1,Y^2$ the coordinates of $X$ and  $Y$, respectively, and consider the sliding vector field $Z^s$ given in  \eqref{svfz}.  In order to study their orbits it is convenient to define the normalized sliding vector field
\begin{equation*}
N(x,0)=(Y h(x,0)- Xh(x,0)) Z^s(x,0)=X^1(x,0) Y^2(x,0) - Y^1(x,0) X^2 (x,0),
\end{equation*}
which has the same phase portrait of $Z^s $ reversing the direction of the trajectory in the escaping region.  Indeed,  $Y h(x,0)- Xh(x,0)$ is positive (resp. negative) for $(x,0)\in\Sigma^s$ (resp. $(x,0)\in\Sigma^e$).

In addition, it is possible to reduce the neighborhood $\V_0$ once more, so that if $Z\in \V_0$, then the following inequalities hold for  $(x,0)$ in a small neighborhood of $(0,0)\in\R^2$:
\begin{equation}
\label{p2}
\begin{array}{ll}
X^2(x,0) <0 ~\text{for} ~ x<0, & X^2(x,0) > 0 ~ \text{for} ~  x> 0, \vspace{0.2cm}\\
Y^2(x,0) <0 ~ \text{for} ~  x< \alpha, & Y^2(x,0) > 0 ~  \text{for} ~ x> \alpha.
\end{array}
\end{equation}

Since  $\alpha<0$, it follows from \eqref{p1} and \eqref{p2} that
\[
\begin{array}{l}
N(\alpha,0) = X^1(\alpha,0) Y^2(\alpha,0)   -   Y^1(\alpha,0) X^2 (\alpha,0) <0,~\text{and} \vspace{0.2cm}\\
N(0,0) = X^1(0,0) Y^2(0,0) - Y^1(0,0) X^2 (0,0) > 0.
\end{array}
\]
Hence $N$ has at least one critical point $(p,0)$ satisfying $\alpha < p < 0$. The unicity of  $p$ will follow by showing that $\dis \frac{\p N}{\p x}(x,0) >0$ for $(x,0)\in \Sigma^s$. So we compute
\begin{equation*}
\begin{array}{ll}
\dis \frac{\p N}{\p x}(x,0) =
&  \dis\frac{\p X^1}{\p x}(x,0) Y^2(x,0) + X^1(x,0) \frac{\p Y^2}{\p x}(x,0)  \vspace*{0.2cm}\\ 
& -  \dis\frac{\p Y^1}{\p x}(x,0) X^2(x,0) - Y^1(x,0) \frac{\p X^2}{\p x}(x,0).
\end{array}
\end{equation*}
Since $X^2(0,0)= Y^2(\alpha,0)=0$, the conditions \eqref{p1}  imply that
$$
\sgn \bigg( \frac{\p N}{\p x}(x,0) \bigg) = \sgn \bigg( X^1(x,0) \frac{\p Y^2}{\p x}(x,0)  - Y^1(x,0) \frac{\p X^2}{\p x}(x,0)\bigg) > 0,
$$
for $(x,0)\in \Sigma^s$. Moreover, $\dfrac{\p N}{\p x}(p,0) >0$ implies that $(p,0)$ is stable.

We conclude this proof  by noticing that $\Sigma^s\cup\Sigma^e\subset \mathcal{A}_0 \cap \Sigma$ provided that $\mathcal{V}_0$ is sufficiently small. The proof for $\alpha>0$ is analogous.
\end{proof}

\section{Proof of Theorem \ref{R2}}
\label{proofA}
Let $\mathcal{A}_0$ and $\V_0$ be, respectively, the annulus of $\Gamma_0$ and the neighborhood of $Z_0$ given by Proposition \ref{fz}. Assume that $\Gamma$ is an elementary simple two-fold cycle of some $Z^*\in\V_0$. Consequently, condition $(C)$ holds for $Z^*$. Therefore, $\eta(Z^*)=(0,0)$. Reciprocally, assume that $\eta(Z^*)=(0,0)$. Therefore, condition $(C)$ holds for $Z^*$, which implies that $Z^*$ admits a simple two-fold cycle $\Gamma\subset \mathcal{A}_0$. Moreover, from Proposition \ref{fz}
\begin{equation}\label{fzstar}
f_{Z^*}(x) = (\ell (Z^*)-k (Z^*))x^2 + \CO(x^3).
\end{equation}
It remains to prove that this cycle is elementary. Note that $\ell (Z_0)-k (Z_0)=M\neq0$. So from the continuity of the functions $\ell$ and $k$ the neighborhood $\V_0$ of $Z_0$ can be reduced in order to guarantee that $\ell (Z)-k (Z)\neq0$ for every $Z\in\V_0$. In particular $\ell (Z^*)-k (Z^*)\neq0$. Hence, from \eqref{fzstar}, $\Gamma$ is elementary. This proves item $(i)$.

Now we shall prove that, for every $Z\in\V_0$, $d\eta(Z)$ is onto. It is equivalent to prove that, for each $(u,v)\in \R^2$, there exists a smooth curve $\CZ:(-\la_0,\la_0)\rightarrow\Omega$ such that $\CZ(0)=Z$ and
\[
d\eta(Z)\cdot\CZ'(0)=\dfrac{d }{d \la}\eta(\CZ(\la))\big|_{\la=0}=(u,v).
\]
Accordingly, let $\CZ(\la)$, $\la\in\R$, be a curve in $\Omega$ such that $\CZ(0)=Z$, $\al(\CZ(\la))=\al(Z)+\la u$, and $\beta(\CZ(\la))=\beta(Z)+\la v$. Therefore,
\[
\dfrac{d }{d \la}\eta(\CZ(\la))\big|_{\la=0}=\lim_{\la\to 0}\dfrac{\eta(\CZ(\la))-\eta(Z)}{\la}=\lim_{\la\to 0}\dfrac{(\la u,\la v)}{\la}=(u,v).
\]
This concludes the proof of Theorem \ref{R2}.

\section{Proof of Theorem \ref{T:BD}}\label{proofB} 
Let $\V_0$ be the neighborhood of $Z_0$ given by Lemma \ref{R1}, and take $\V_1 =\eta^{-1}(U_1)$, where $U_1$ is the neighborhood given by Lemma \ref{lemzeros}. Note that $\V_1 \subset \V_0$, because $U_1  \subset U=\eta(\V_0)$.  In parallel let $\mathcal{A}_0$ be the annulus of $\Gamma_0$ given by Theorem \ref{R2}, and take a annulus $\mathcal{A}_1\subset \mathcal{A}_0$ such that $(\mathcal{A}_1 \cap \Sigma) \setminus \ov{\Sigma^s} \subset \Sigma^c$. This choice for $\mathcal{A}_1$ assures  that the dynamics in $\mathcal{A}_1$ of a vector field $Z\in \V_1$ is completely determined by the function $\eta$.  Now we are able to prove the Theorem \ref{T:BD}.

For $Z\in \V_1$, consider the pseudo equilibrium points given by Lemma \ref{R1}: $(p^u,0)$ for $\al<0$, and $(p^s,0)$, for $\al>0$. From the continuous dependence of the solutions on the parameters $\al$ and $\beta$, there exist two curves $\beta_4(\al)$, defined for $\al>0$, and $\beta_5(\al)$, defined for $\al<0$, for which the following statements hold:
\begin{itemize}
	\item  when $\beta=\beta_4(\al)$  there exists an orbit connecting the fold point  $(\alpha,0)$ and the stable pseudo equilibrium $(p^s,0)$;
	
	\item  when  $\beta=\beta_5(\al)$ there exists an orbit connecting the fold point $(0,0)$ and the unstable pseudo equilibrium $(p^u,0)$.
\end{itemize}
In additional, $0<\beta_4(\al)<\beta_2(\al)$ for $\alpha>0$, and $\beta_3(\al)<\beta_5(\al)<0$ for  $\alpha <0$.

Now consider the regions $R_i$, defined in \eqref{regions}, and the boundaries $B_{ij} = \p R_i \cup \p R_j \setminus \{(0,0)\}$ for $i,j \in \{ 1,2, \ldots, 9\}$.  In order to prove Theorem \ref{T:BD}, we shall use the above results to describe every element of the bifurcation diagram.

Let $Z\in \V_1$ and $\eta(Z)=(\alpha,\beta)$. From the definition of the function $\eta$ we have:
\begin{itemize}
	\item[$(1)$]  two distinct regular fold points, $(0,0)$ and $(\alpha,0)$, respectively, when $\alpha \neq 0,$ that is, $(\al,\beta)\in (R_1\cup R_2 \cup \ov{R_3}\cup \ov{R4}\cup \ov{R_5}\cup{R_6} \cup R_7 \cup \ov{R8} \cup \ov{R9})\setminus\{(0,0)\}$;
	
	\item[$(2)$] a two-fold point at $(0,0)$ if $\alpha=0,$ that is, $(\al,\beta)\in B_{12} \cup B_{67}\setminus\{(0,0)\}$;
	
	\item[$(3)$] and a connection between  $(0,0)$  and $(\alpha,0)$ if $\beta=0,$ that is, $(\al,\beta)\in B_{23} \cup B_{78}$.
\end{itemize}
Moreover, from Lemma \ref{R1} we also conclude that there exist:
\begin{itemize}
   \item[$(4)$] a stable pseudo equilibrium $(p^s,0) \in \Sigma^e$ and two regular fold points, $(0,0)$ and $(\alpha,0)$, with $0 < p^s < \alpha$, when $\alpha>0,$ that is, $(\al,\beta)\in (R_2\cup \ov{R_3}\cup \ov{R4}\cup\ov{R_5}\cup R_6)\setminus\{(0,0)\}$;
	
   \item[$(5)$] a unstable pseudo equilibrium $(p^u,0) \in \Sigma^s$ and two regular fold points, $(0,0)$ and $(\alpha,0)$, with $\alpha < p^u < 0$, when $\alpha<0,$ that is, $(\al,\beta)\in (R_7\cup \ov{R_8}\cup R_9)\setminus\{(0,0)\}$;

   \item[$(6)$] a sliding connection between $(0,0)$ and $(\alpha,0)$, when $\alpha >0$ and $ 0 < \beta < \beta_4(\alpha)$ or $\alpha <0$ and $ \beta_5(\alpha) < \beta < 0,$ that is, $(\al,\beta)\in (R_3\cup R_8)$;
	
   \item[$(7)$] a connection between $(p^s,0)$ and $(\alpha,0)$ when $\beta=\beta_4(\alpha),$ that is,  $(\al,\beta)\in B_{34}$;
   
   \item[$(8)$] a connection between $(0,0)$ and $(p^u,0)$ when $\beta=\beta_5(\alpha),$ that is,  $(\al,\beta)\in B_{89}$;	
	
   \item[$(9)$] a unstable sliding cycle connection passing through $(\alpha,0)$ when $\alpha>0$ and $\beta_4(\alpha) < \beta_2(\alpha),$ that is,  $(\al,\beta)\in R_4$;	
   
   \item[$(10)$] and a stable sliding cycle connection passing through $(0,0)$ when $\alpha<0$ and $\beta_3(\alpha) < \beta_5(\alpha),$ that is,  $(\al,\beta)\in R_9$.	
\end{itemize}
Furthermore, from Remark \ref{defi}, and Lemmas \ref{lemzeros} and \ref{lemzeros2}  we conclude that there exist:
\begin{itemize}
\item[$(11)$] a stable crossing cycle $\gamma_1$ passing through $(x^+,0)$, with $x^+> \text{max}\{0,\alpha\}$,  when $\alpha>0$ and $\beta <\beta_1(\alpha)$ or $\alpha \leq 0$ and $\beta <\beta_3(\alpha),$ that is, $(\al,\beta)\in (R_1\cup \ov{R_2}\cup\cdots\cup\ov{R_4}\cup R_5)\setminus\{(0,0)\}$; 
	
\item[$(12)$] a unstable critical crossing cycle passing through $(x^-,0)$, with $x^-=\alpha$, when $\beta=\beta_2(\alpha)$, that is, $(\al,\beta)\in B_{45}$ ; 
	
\item[$(13)$] a unstable crossing cycle $\gamma_2$ passing through $(x^-,0)$ when $\alpha>0$ and $\beta_2(\alpha) < \beta < \beta_1(\alpha),$ that is, $(\al,\beta)\in R_5$. In this case, $x^- < x^+$ such that $\gamma_2 \subset ext (\gamma_1)$; 	 
	
\item[$(14)$] a semi-stable crossing cycle $\gamma_3$ passing through $(x^*,0)$ when $\alpha>0$ and $\beta =\beta_1(\alpha),$ that is, $(\al,\beta)\in B_{56}$.  This cycle $\gamma_3$ comes from the collapse between the cycles $\gamma_1$ and $\gamma_2$, since for $\beta =\beta_1(\alpha)$ we have $x^\pm=x^*$;
	
\item[$(15)$] and a stable critical crossing cycle passing through $(x^+,0)$, with $x^+=0$,  when $\beta= \beta_3(\alpha),$ that is, $(\alpha,\beta)\in B_{91}$.
\end{itemize}

This concludes the proof of Theorem \ref{T:BD}.

\section{Piecewise Hamiltonian Example}\label{ppm}

In this section we present a 2-parameter family of piecewise Hamiltonian vector fields $Z_{a,b}$ realizing the bifurcation diagram described by Figure  \ref{stable}. Let $U\subset\R^2$ be a small neighborhood of the origin $(0,0)$, and, for $(a,b)\in U$, consider the following piecewise polynomial vector field:
\begin{equation*}\label{example}
Z_{a,b}(x,y)=\left\{\begin{array}{lr} 
X_{a,b}(x,y)=
\left(\begin{array}{c}
1-y\\
x-8x^3
\end{array}\right) & \textrm{if} ~ y>0,\vspace{0.2cm}\\
Y_{a,b}(x,y)=
\left(\begin{array}{c}
-1-y\\
(x-a)(1+a-2b-3x)
\end{array}\right)  & \textrm{if}~ y<0.
\end{array}\right.
\end{equation*} 
Notice that $X_{a,b}$ and $Y_{a,b}$ are Hamiltonian vector fields. Indeed, for  
\begin{equation*}\label{hamil}
\begin{array}{l}
H^+_{a,b}(x,y)= 2x^4-\dfrac{x^2}{2}+y-\dfrac{y^2}{2} \text{ and} \vspace{0.2cm}\\
H^-_{a,b}(x,y)=x^3-\dfrac{(1+4a-2b)x^2}{2}+a(1+a-2b)x-\dfrac{y^2}{2}-y,
\end{array}
\end{equation*}
we have
\[
X_{a,b}=\left(\dfrac{\p H^+_{a,b}}{\p y},-\dfrac{\p H^+_{a,b}}{\p x}\right)\text{ and } Y_{a,b}=\left(\dfrac{\p H^-_{a,b}}{\p y},-\dfrac{\p H^-_{a,b}}{\p x} \right).
\]

It is easy to see that $X_{a,b}$ has a visible fold point at $(0,0)$ and $Y_{a,b}$ has a visible fold point at $(a,0)$, so that the piecewise Hamiltonian vector field $Z_{a,b}$ has a visible two-fold singularity at $(0,0)$, when $a=0$. Notice that $\al(Z_{a,b})=a$. 

Given $x_0>0$ small, denote by $(\xi^+_{a,b}(x_0),0)$ the first intersection of the forward trajectory of $X_{a,b}$ passing through $(x_0,0)$ with $\Sigma$, and by  $(\xi^-_{a,b}(x_0),0)$ the first intersection of the backward trajectory of $Y_{a,b}$ passing through $(x_0,0)$ with $\Sigma$. Using that $\xi^+_{a,b}(x_0)$ and $\xi^-_{a,b}(x_0)$ satisfy
\[
H^+_{a,b}(\xi^+_{a,b}(x_0),0)=H^+_{a,b}(x_0,0) \quad \text{and} \quad H^-_{a,b}(\xi^-_{a,b}(x_0),0)=H^-_{a,b}(x_0,0),
\]
we compute
\[
\begin{array}{l}
\xi^+_{a,b}(x_0)=\dfrac{\sqrt{1-4x_0^2}}{2} \quad \text{and}\vspace{0.2cm}\\
\xi^-_{a,b}(x_0)= \dfrac{1+4a-2b-2x_0+\sqrt{(1-8a-2b+6x_0)(1-2b-2x_0)}}{4}.
 \end{array}
\]

Observe that $\be(Z_{a,b})=\xi^+_{a,b}(0)-\xi^-_{a,b}(a)=b$. Therefore, $\eta(Z_{0,0})=(0,0)$ which implies that $Z_{0,0}$ has a simple two-fold cycle $\Gamma_0$. Moreover, the displacement function $f_{Z_{0,0}}$ defined in \eqref{fZ} is given by
\[
\begin{array}{rl}
f_{Z_{0,0}}(x_0)=&\xi^+_{a,b}(x_0)-\xi^-_{a,b}(x_0)\vspace{0.2cm}\\
=& \dfrac{-1+2x_0-\sqrt{1+4(1-3x_0)x_0}+2\sqrt{1-4x_0^2}}{4}\vspace{0.2cm} \\
=& x_0^2+\CO(x_0^3).
\end{array}
\]
Hence $\Gamma_0$ is a stable elementary simple two-fold cycle. Finally, since $\eta(Z_{a,b})=(a,b)$, the unfolding of $\Gamma_0$ is given by Theorem \ref{T:BD}, which means that the 2-parameter family of piecewise Hamiltonian systems $Z_{a,b}$ realizes, for $(a,b)\in U$, the bifurcation diagram described by Figure \ref{stable}.

\section{Conclusion and Further Directions}\label{conclusion}

The primary goal of this work has been to present a qualitative and geometric analysis of the generic unfolding of a codimension-two cycle. Such object, named elementary simple two-fold cycle, is characterized by a regular trajectory connecting a two-fold singularity to itself for which it is well defined a first return map with nonvanishing second derivative at the two-fold.  We analyzed the codimension-two scenario through the exhibition of its bifurcation diagram, that is, loci of codimension-one phenomena, in the parameter space, emanating from the origin which corresponds to the elementary simple two-fold cycle. Each locus of codimension-one phenomenon corresponds to either a saddle node bifurcation of limit cycles, a critical crossing bifurcation, a connection between fold-regular singularities, or a connection between pseudo equilibrium and regular-fold singularity. Finally, we provided a 2-parameter family of piecewise mechanical systems realizing the bifurcation diagram.

Regarding future works, a first possible direction is to consider polycycles of planar Filippov systems. A polycycle is a simple closed curve composed by finitely many singularities connected by regular trajectories for which a first return map is well defined. The same qualitative aspects that have been investigated in this paper for simple two-fold cycles can also be considered for polycycles in planar Filippov systems. There are also many open questions  concerning the dynamics of piecewise differential systems in higher dimensions. The two-fold singularity has already been considered in higher dimensional nonsmooth differential systems (see, for instance, \cite{CJ,T82,T90,T93}), so it is natural to inquire about  two-fold cycles in this context.  
Since a two-fold singularity is generically not an isolated singularity in higher dimension, some extra difficulties may naturally appear in such approach. Moreover, in the $n$-dimensional case, generic conditions may be assumed in order to guarantee that a simple two-fold cycle has codimension one when $n=3$, and codimension zero when $n>3$.

\section*{Acknowledgments}

We thank to the referees for their comments and suggestions which helped us to improve greatly the presentation of this paper.

DDN is partially supported by a FAPESP grant 2016/11471-2. MAT is partially supported by a FAPESP grant 2012/18780-0 and by a CNPq grant 300596/ 2009-0. IOZ is partially supported by FAPESP grants 2012/23591-1 and 2013/21078-8.

\bibliographystyle{abbrv}
\bibliography{references.bib}

\end{document}